\documentclass[oneside]{amsart}
\usepackage[latin9]{inputenc}
\usepackage{amsthm}
\usepackage{graphicx}

\usepackage{amsmath, amsthm, amssymb,mathtools,slashbox,placeins,aliascnt,setspace,fancybox,hyperref}
\usepackage[all]{hypcap}
\usepackage{fullpage}

\theoremstyle{plain}
\newtheorem{thm}{Theorem}

\theoremstyle{definition}
\newaliascnt{defn}{thm}
\newtheorem{defn}[defn]{Definition}
\aliascntresetthe{defn}

\theoremstyle{remark}
\newaliascnt{rem}{thm}

\aliascntresetthe{rem}

\theoremstyle{plain}
\newaliascnt{fact}{thm}
\newtheorem{fact}[fact]{Fact}
\aliascntresetthe{fact}

\theoremstyle{plain}
\newaliascnt{lem}{thm}
\newtheorem{lem}[lem]{Lemma}
\aliascntresetthe{lem}

\theoremstyle{plain}
\newaliascnt{prop}{thm}
\newtheorem{prop}[prop]{Proposition}
\aliascntresetthe{prop}

\theoremstyle{plain}
\newaliascnt{proc}{thm}
\newtheorem{proc}[proc]{Procedure}
\aliascntresetthe{proc}

\theoremstyle{plain}
\newaliascnt{cor}{thm}
\newtheorem{cor}[cor]{Corollary}
\aliascntresetthe{cor}

\theoremstyle{definition}
\newaliascnt{example}{thm}

\aliascntresetthe{example}

\theoremstyle{plain}
\newaliascnt{conv}{thm}
\newtheorem{conv}[conv]{Convention}
\aliascntresetthe{conv}

\def\equationautorefname~#1\null{(#1)\null}

\newcommand{\A}{\mathcal A} 
\newcommand{\B}{\mathcal B} 
\newcommand{\C}{\mathcal C} 
\newcommand{\D}{\mathcal D} 
\newcommand{\F}{\mathcal F} 
\renewcommand{\H}{\mathcal H} 

\newcommand{\X}{\mathcal X} 
\newcommand{\Y}{\mathcal Y} 
\newcommand{\Z}{\mathcal Z} 

\newcommand{\K}{\mathcal K} 
\newcommand{\R}{\mathcal R} 
\newcommand{\M}{\mathcal M} 
\renewcommand{\P}{\mathcal P} 
\newcommand{\T}{\mathcal T} 
\newcommand{\U}{\mathcal U} 
\newcommand{\V}{\mathcal V} 
\newcommand{\W}{\mathcal W} 

\newcommand{\AP}{\widetilde{\mathcal A}}

\newcommand{\CP}{\widetilde{\mathcal C}}

\newcommand{\XP}{\widetilde{\mathcal X}}

\newcommand{\ap}{\widetilde a}
\newcommand{\cp}{\widetilde c}

\newcommand{\cb}{\overline c}

\newcommand{\AB}{\overline{\mathcal A}}

\newcommand{\CB}{\overline{\mathcal C}}

\newcommand{\XB}{\overline{\mathcal X}}

\newcommand{\PA}[5]{\mathcal{PA}_{#1,#2;#3}^{\left(#4;#5\right)}}
\newcommand{\CA}[4]{\mathcal{CA}_{#1,#2}^{\left(#3;#4\right)}}

\newcommand{\VA}[4]{\mathcal{VA}_{#1,#2;#3}^{\left(#4\right)}}

\newcommand{\PVA}[4]{\mathcal{PVA}_{#1,#2;#3}^{\left(#4\right)}}

\newcommand{\AR}[4]{\mathcal{AR}_{#1;#2,#3}^{\left(#4\right)}}

\newcommand{\ksub}[2]{\left[#1;#2\right]}

\begin{document}
\let\ref\autoref

\title{Enumeration of Tree-like Maps with Arbitrary Number of Vertices}

\author{Aaron Chun Shing Chan}

\date{December 5, 2016}
\begin{abstract}
This paper provides the generating series for the embedding of tree-like
graphs of arbitrary number of vertices, accourding to their genus.
It applies and extends the techniques of Chan \cite{Chan:2017-1},
where it was used to give an alternate proof of the Goulden and Slofstra
formula. Furthermore, this greatly generalizes the famous Harer-Zagier
formula \cite{Harer-Zagier:1986}, which computes the Euler characteristic
of the moduli space of curves, and is equivalent to the computation
of one vertex maps.
\end{abstract}

\maketitle

\section{\label{sec:Introduction}Introduction}

Let $n$ and $k$ be integers such that $0\le n,k$. We use $\left[n\right]$
to denote the set $\left\{ 1,\dots,n\right\} $, $\left[n\right]^{k}$
to denote the Cartesian product of $\left[n\right]$ with itself $k$
times, and $\ksub{n}{k}$ to denote the set of all $k$-subsets of
$\left[n\right]$. If $S$ is a set of even cardinality, then a \emph{pairing}
$\mu$ of $S$ is a partition of $S$ into disjoint subsets of size
2. Next, a \emph{partial pairing} $T$ of a set $S$ is a pairing
on a subset $S^{\prime}\subseteq S$ of even cardinality. If $\left|S^{\prime}\right|=2k$,
then $T$ is called a \emph{$k$-partial pairing} of $S$. The set
$S^{\prime}$ is called the \emph{support} of the partial pairing
$T$. Finally, the set of all $k$-partial pairings of $\left[n\right]$
is denoted as $\T_{n,k}$, which has cardinality $\left|\T_{n,k}\right|=\binom{n}{2k}\left(2k-1\right)!!$,
where $\left(2k-1\right)!!=\prod_{j=1}^{k}\left(2j-1\right)$ is the
double factorial, with the convention that $\left(-1\right)!!=1$.

Let $p$ and $n$ be positive integers. We use $\left[p\right]^{\underline{n}}$
to denote the set $\left\{ 1^{^{\underline{n}}},2^{^{\underline{n}}},\dots,p^{^{\underline{n}}}\right\} $,
whose elements $i^{\underline{n}}$, $i=1,\dots,p$, are regarded
as a labelled version of the integer $i$, labelled by the ``$\underline{n}$''
in the superscript position. Then, suppose $\mathbf{p}=\left(p_{1},\dots,p_{n}\right)$
is a vector of length $n$ of positive integers, we let $\left[p_{1},\dots,p_{n}\right]$
to be the set $\left[p_{1}\right]^{\underline{1}}\cup\cdots\cup\left[p_{n}\right]^{\underline{n}}$.
For example, $\left[3,5,2\right]$ is the set $\left\{ 1^{^{\underline{1}}},2^{^{\underline{1}}},3^{^{\underline{1}}},1^{^{\underline{2}}},2^{^{\underline{2}}},3^{^{\underline{2}}},4^{^{\underline{2}}},5^{^{\underline{2}}},1^{^{\underline{3}}},2^{^{\underline{3}}}\right\} $.
Furthermore, if $p_{1}+\cdots+p_{n}$ is even, then the set of all
pairings of $\left[p_{1},\dots,p_{n}\right]$ is denoted as $\P_{p_{1},\dots,p_{n}}$.
Now, if $\mu$ is a pairing of $\left[p_{1},\dots,p_{n}\right]$,
then a pair $\left\{ x^{\underline{i}},y^{\underline{k}}\right\} $
in $\mu$ is a \emph{mixed pair} if $i\neq k$, and a \emph{non-mixed
pair} otherwise. To describe the number of mixed and non-mixed pairs
in a pairing $\mu$, we introduce the parameters $\mathbf{q}$ and
$\mathbf{s}$. Let $\mathbf{q}=\left(q_{1},\dots,q_{n}\right)$ be
a vector of length $n$, and $\mathbf{s}=\left(s_{1,2},s_{1,3},\dots,s_{n-1,n}\right)$
be an $n\times n$ strictly upper triangular matrix, where for ease
of notation we let $s_{i,k}=s_{k,i}$ for $i>k$ and $s_{i}=\sum_{k\neq i}s_{i,k}$.
If $p_{i}=2q_{i}+s_{i}$ is positive for $1\le i\le n$, we define
$\P_{n}^{\left(\mathbf{q};\mathbf{s}\right)}\subseteq\P_{p_{1},\dots,p_{n}}$
to be the subset of the pairing such that for $\mu\in\P_{n}^{\left(\mathbf{q};\mathbf{s}\right)}$,
$\mu$ has $q_{i}$ non-mixed pairs of the form $\left\{ x^{\underline{i}},y^{\underline{i}}\right\} $
and $s_{i,k}$ mixed pairs of the form $\left\{ x^{\underline{i}},y^{\underline{k}}\right\} $.
When convenient, we will sometimes treat $\mathbf{s}$ as a vector
of length $\frac{n\left(n-1\right)}{2}$. Furthermore, the \emph{support
graph} of $\mathbf{s}$ is the graph $G$ with the vertex set $\left[n\right]$,
such that $\left\{ i,k\right\} $ is an edge of $G$ if and only if
$s_{i,k}>0$.

Let $\gamma_{p_{1},\dots,p_{n}}$ be the canonical cycle permutation
of $\P_{p_{1},\dots,p_{n}}$, given by $\gamma_{p_{1},\dots,p_{n}}=\left(1^{^{\underline{1}}},\dots,p_{1}^{^{\underline{1}}}\right)\cdots\left(1^{^{\underline{n}}},\dots,p_{n}^{^{\underline{n}}}\right)$.
For $L\ge1$, we define $\A_{n,L}^{\left(\mathbf{q};\mathbf{s}\right)}\subseteq\P_{n}^{\left(\mathbf{q};\mathbf{s}\right)}$
to be the subset of pairings such that for $\mu\in\A_{n,L}^{\left(\mathbf{q};\mathbf{s}\right)}$,
$\mu\gamma_{p_{1},\dots,p_{n}}^{-1}$ has exactly $L$ cycles, and
let $a_{n,L}^{\left(\mathbf{q};\mathbf{s}\right)}=\left|\A_{n,L}^{\left(\mathbf{q};\mathbf{s}\right)}\right|$.
Our result can be stated as follows
\begin{thm}
\label{thm:main formula} Let $n\ge1$, $\mathbf{q}=\left(q_{1},\dots,q_{n}\right)$
and $\mathbf{s}=\left(s_{1,2},s_{1,3},\dots,s_{n-1,n}\right)$ be
vectors of non-negative integers, and suppose that the support graph
$G$ of $\mathbf{s}$ is a tree with edges $e_{1},\dots,e_{n-1}$.
Then, the generating series $A_{n}^{\left(\mathbf{q};\mathbf{s}\right)}\left(x\right)=\sum_{L\ge1}a_{n,L}^{\left(\mathbf{q};\mathbf{s}\right)}x^{L}$
satisfies
\begin{eqnarray*}
A_{n}^{\left(\mathbf{q};\mathbf{s}\right)}\left(K\right) & = & \sum_{\mathbf{t}=\mathbf{0}}^{\mathbf{q}}\prod_{i=1}^{n}\frac{\left(2q_{i}+s_{i}\right)!}{2^{t_{i}}t_{i}!\left(s_{i}+q_{i}-t_{i}\right)!}\cdot v_{n,K;\mathbf{q-t+1}}^{\left(\mathbf{s}\right)}
\end{eqnarray*}
for all $K\ge1$, where 
\begin{eqnarray*}
v_{n,K;\mathbf{R}}^{\left(\mathbf{s}\right)} & = & \sum_{A_{e_{1}}=0}^{\min\left(s_{e_{1}},K\right)-1}\cdots\sum_{A_{e_{n-1}}=0}^{\min\left(s_{e_{n-1}},K\right)-1}\left[\prod_{j=1}^{n-1}\frac{\left(K-A_{e_{j}}-1\right)!}{\left(K+s_{e_{j}}-A_{e_{j}}-1\right)!}\times\right.\\
 &  & \left.\prod_{i=1}^{n}\frac{\left(K+\sum_{k\sim i}\left(s_{i,k}-A_{i,k}-1\right)\right)!\left(R_{i}-1+\sum_{k\sim i}s_{i,k}\right)!}{\left(R_{i}-1\right)!\left(K-R_{i}-\sum_{k\sim i}A_{i,k}\right)!\left(R_{i}+\sum_{k\sim i}\left(s_{i,k}-1\right)\right)!}\right]
\end{eqnarray*}
Furthermore, for fixed $n$, $\mathbf{q}$, and $\mathbf{s}$, this
expression can be written as a polynomial in $K$.
\end{thm}
In this expression, the sum $\sum_{k\sim i}$ is over all indices
$k$ that are adjacent to $i$ in the support graph of $\mathbf{s}$.
Furthermore, for each edge $e_{j}=\left\{ i,k\right\} $, the summation
variable $A_{e_{j}}$ is equivalently written as $A_{i,k}$ and $A_{k,i}$
in parts of the expression. Now, the fact that this expression can
be written as a polynomial in $K$ for fixed parameters $n$, $\mathbf{q}$,
and $\mathbf{s}$ means that we can substitute $K=x$ into our expression
for $A_{n}^{\left(\mathbf{q};\mathbf{s}\right)}\left(K\right)$ to
obtain $A_{n}^{\left(\mathbf{q};\mathbf{s}\right)}\left(x\right)$.

In the language of enumerating maps, this generating series counts
the number of combinatorial maps with $n$ vertices and $L$ faces,
such that there are $q_{i}$ loop edges incident to vertex $i$, and
$s_{i,j}$ edges between vertices $i$ and $j$. Furthermore, the
combinatorial maps counted in this series are connected if and only
if the support graph of $\mathbf{s}$ is connected. A survey on the
relationship between maps and the products of permutations can be
found in \cite{Lando-Zvonkin:2004}.

This theorem generalizes a number of theorems already existing in
the literature. In particular, the $n=1$ case of our theorem is the
Harer-Zagier formula for computing the Euler characteristic of the
moduli space of curves, which can be written as follows.
\begin{thm}
(Harer-Zagier \cite{Harer-Zagier:1986}) \label{thm:Harer-Zagier}
Let $q$ be a positive integer, and $\A_{L}^{\left(q\right)}$ be
the subset of pairings of $\P_{2q}$ such that for $\mu\in\A_{L}^{\left(q\right)}$,
$\mu\gamma_{2q}^{-1}$ has exactly $L$ cycles. If we let $a_{L}^{\left(q\right)}=\left|\A_{L}^{\left(q\right)}\right|$,
then the generating series for $a_{L}^{\left(q\right)}$ is given
by 
\[
A^{\left(q\right)}\left(x\right)=\left(2q-1\right)!!\sum_{k\ge1}2^{k-1}\binom{q}{k-1}\binom{x}{k}
\]
\end{thm}
There are numerous proofs of this formula in the literature, both
algebraic and combinatorial. A selection of the proofs can be found
in the papers by Goulden and Nica \cite{Goulden-Nica:2005}, Itzykson
and Zuber \cite{Itzykson-Zuber:1990}, Jackson \cite{Jackson:1994},
Kerov \cite{Kerov:1999}, Kontsevich \cite{Kontsevich:1992}, Lass
\cite{Lass:2001}, Penner \cite{Penner:1988}, and Zagier \cite{Zagier:1995}.
The original proof of Harer-Zagier uses matrix integration, and there
are numerous other algebraic proofs for this same result. Some subsequent
proofs used purely combinatorial approaches, such as the use of Eulerian
tours by Lass, and the use of trees by Goulden and Nica. To reduce
\ref{thm:main formula} to the Harer-Zagier formula, we can simply
take $q_{1}=q$, $\mathbf{s}$ to be empty, and then reversing the
sum with $t=q-k-1$.

The $n=2$ case of our theorem was proved by Goulden and Slofstra
\cite{Goulden-Slofstra:2010} using a combinatorial technique that
we will extend in this paper.
\begin{thm}
(Goulden-Slofstra \cite{Goulden-Slofstra:2010}) \label{thm:Goulden-Slofstra}
Let $q_{1}$ and $q_{2}$ be non-negative integers, and $s$ be a
positive integer. Let $\A_{L}^{\left(q_{1},q_{2};s\right)}$ be the
subset of pairings of $\P^{\left(q_{1},q_{2};s\right)}$ such that
for $\mu\in\A_{L}^{\left(q_{1},q_{2};s\right)}$, $\mu\gamma_{2q_{1}+s,2q_{2}+s}^{-1}$
has exactly $L$ cycles. If we let $a_{L}^{\left(q_{1},q_{2};s\right)}=\left|\A_{L}^{\left(q_{1},q_{2};s\right)}\right|$,
then the generating series for $a_{L}^{\left(q_{1},q_{2};s\right)}$
is given by 
\[
A^{\left(q_{1},q_{2};s\right)}\left(x\right)=p_{1}!p_{2}!\sum_{k=1}^{d+1}\sum_{i=0}^{\left\lfloor \frac{1}{2}p_{1}\right\rfloor }\sum_{j=0}^{\left\lfloor \frac{1}{2}p_{2}\right\rfloor }\frac{1}{2^{i+j}i!j!\left(d-i-j\right)!}\binom{x}{k}\binom{d-i-j}{k-1}\Delta_{k}^{\left(q_{1},q_{2};s\right)}
\]
where $p_{1}=2q_{1}+s$, $p_{2}=2q_{2}+s$, $d=q_{1}+q_{2}+s$, and
\[
\Delta_{k}^{\left(q_{1},q_{2};s\right)}=\binom{k-1}{q_{1}-i}\binom{k-1}{q_{2}-j}-\binom{k-1}{q_{1}+s-i}\binom{k-1}{q_{2}+s-j}
\]
\end{thm}
In this expression, $p_{1}$ and $p_{2}$ are the degrees of vertices
1 and 2, respectively, and $d$ is the total number of pairs in the
pairing.

Unlike the $n=1$ case, the most direct way to show that \ref{thm:main formula}
can be reduced to \ref{thm:Goulden-Slofstra} is to delve into the
combinatorial proof itself. By noting the differences in the definitions
of vertical arrays between Goulden and Slofstra and our subsequent
definitions, we can relate their cardinalities using inclusion-exclusion.
Further algebraic manipulations then shows that the two formulas are
equivalent. As the proof is rather lengthy, readers interested in
the proof can consult \cite{ChanThesis:2016}.

\section{\label{sec:Paired Arrays}Paired Functions and Paired Arrays}

For proving our main theorem, we will use a combinatorial object called
\emph{paired functions}, which are related to the paired surjections
introduced in Goulden and Slofstra \cite{Goulden-Slofstra:2010}.
The difference between the two objects is that we reject the non-empty
condition here, which makes our object equivalent to the $K$-colouring
cycles used in some of the algebraic techniques in \cite{Lando-Zvonkin:2004}.
This brings together the algebraic and combinatorial techniques, as
they effectively count the same set of objects.
\begin{defn}
\label{def:Paired Functions}Let $n,K\ge1$, $\mathbf{q}=\left(q_{1},\dots,q_{n}\right)\ge\mathbf{0}$,
$\mathbf{s}=\left(s_{1,2},s_{1,3},\dots,s_{n-1,n}\right)\ge\mathbf{0}$,
and $p_{i}=2q_{i}+\sum_{k\neq i}s_{k,i}$ for $1\le i\le n$. An ordered
pair $\left(\mu,\pi\right)$ is a \emph{paired function} if $\mu\in\P_{n}^{\left(\mathbf{q};\mathbf{s}\right)}$
and $\pi\colon\left[p_{1},\dots,p_{n}\right]\rightarrow\left[K\right]$
is a function satisfying 
\[
\pi\left(\mu\left(v\right)\right)=\pi\left(\gamma_{p_{1},p_{2},\dots,p_{n}}\left(v\right)\right)\quad\mbox{ for all }v\in\left[p_{1},\dots,p_{n}\right]
\]
We denote the set of paired functions satisfying the parameters $n$,
$K$, $\mathbf{q}$, and $\mathbf{s}$ as $\F_{n,K}^{\left(\mathbf{q};\mathbf{s}\right)}$,
and we let $f_{n,K}^{\left(\mathbf{q};\mathbf{s}\right)}=\left|\F_{n,K}^{\left(\mathbf{q};\mathbf{s}\right)}\right|$.
\end{defn}
By substituting in $u=\gamma_{p_{1},p_{2},\dots,p_{n}}\left(v\right)$,
we have $\pi\left(u\right)=\pi\left(\mu\gamma_{p_{1},p_{2},\dots,p_{n}}^{-1}\left(u\right)\right)$
for all $u\in\left[p_{1},\dots,p_{n}\right]$. This implies that the
cycles of $\mu\gamma_{p_{1},p_{2},\dots,p_{n}}^{-1}$ are preserved
by $\pi$. In other words, each of the cycles of $\mu\gamma_{p_{1},p_{2},\dots,p_{n}}^{-1}$
is coloured with one of $K$ colours. Hence, for any given pairing
$\mu\in\A_{n,L}^{\left(\mathbf{q};\mathbf{s}\right)}$, there are
$K^{L}$ functions $\pi\colon\left[p_{1},\dots,p_{n}\right]\rightarrow\left[K\right]$
such that $\left(\mu,\pi\right)$ is a paired function. Furthermore,
by applying the definition to all pairs $\left\{ x^{\underline{i}},y^{\underline{k}}\right\} $
of $\mu$, we have that $\left(\mu,\pi\right)$ is a paired function
if and only if 
\begin{eqnarray}
\left(\pi\left(\mu\left(y^{\underline{k}}\right)\right),\pi\left(\gamma_{p_{1},p_{2},\dots,p_{n}}\left(x^{\underline{i}}\right)\right)\right) & = & \left(\pi\left(\gamma_{p_{1},p_{2},\dots,p_{n}}\left(y^{\underline{k}}\right)\right),\pi\left(\mu\left(x^{\underline{i}}\right)\right)\right)\nonumber \\
\left(\pi\left(x^{\underline{i}}\right),\pi\left(\left(x+1\right)^{\underline{i}}\right)\right) & = & \left(\pi\left(\left(y+1\right)^{\underline{k}}\right),\pi\left(y^{\underline{k}}\right)\right)\label{eq:Algebraic Relation}
\end{eqnarray}
holds for all pairs $\left\{ x^{\underline{i}},y^{\underline{k}}\right\} $
of $\mu$, where addition is done modulo $p_{i}$ and $p_{k}$ on
the left and right hand side, respectively.

Recall that $a_{n,L}^{\left(\mathbf{q};\mathbf{s}\right)}$ is the
number of pairings $\mu\in\P_{n}^{\left(\mathbf{q};\mathbf{s}\right)}$
such that $\mu\gamma_{p_{1},p_{2},\dots,p_{n}}^{-1}$ has exactly
$L$ cycles. Hence, for each pairing $\mu$, there are $K^{L}$ functions
$\pi$ such that $\left(\mu,\pi\right)$ is a paired function. This
gives us 
\begin{equation}
A_{n}^{\left(\mathbf{q};\mathbf{s}\right)}\left(K\right)=\sum_{L\ge1}a_{n,L}^{\left(\mathbf{q};\mathbf{s}\right)}K^{L}=f_{n,K}^{\left(\mathbf{q};\mathbf{s}\right)}\label{eq:Genfunc Label Relation}
\end{equation}
 for $K\ge1$. Therefore, if we can find an expression for $f_{n,K}^{\left(\mathbf{q};\mathbf{s}\right)}$
that is a polynomial in $K$, we can substitute $K=x$ into that expression
to obtain $A_{n}^{\left(\mathbf{q};\mathbf{s}\right)}\left(x\right)$.

To represent paired functions, we use a graphical representation introduced
in Goulden and Slofstra, called the \emph{labelled array}. This is
an $n\times K$ array of cells arranged in a grid. Each element $x^{\underline{i}}$
of $\mu$ is represented as a vertex, where the vertex labelled $x^{\underline{i}}$
is placed into cell $\left(i,j\right)$ if $\pi\left(x^{\underline{i}}\right)=j$.
The vertices are arranged horizontally within a cell, in increasing
order of the labels. Furthermore, for each pair $\left\{ x^{\underline{i}},y^{\underline{k}}\right\} $
in $\mu$, an edge is drawn between their corresponding vertices.

For example, let $\left(\mu,\pi\right)\in\F_{3,4}^{\left(\mathbf{q};\mathbf{s}\right)}$,
where $\mathbf{q}=\left(2,2,3\right)$, and $\mathbf{s}=\left(1,3,1\right)$.
Suppose $\mu$ and $\pi$ are given by 
\begin{eqnarray*}
\mu & = & \left\{ \left\{ 1^{^{\underline{1}}},2^{^{\underline{1}}}\right\} ,\left\{ 3^{^{\underline{1}}},10^{^{\underline{3}}}\right\} ,\left\{ 4^{^{\underline{1}}},9^{^{\underline{3}}}\right\} ,\left\{ 5^{^{\underline{1}}},4^{^{\underline{2}}}\right\} ,\left\{ 6^{^{\underline{1}}},7^{^{\underline{1}}}\right\} ,\left\{ 8^{^{\underline{1}}},1{}^{^{\underline{3}}}\right\} \right.\\
 &  & \left.\,\,\:\left\{ 1^{^{\underline{2}}},6^{^{\underline{2}}}\right\} ,\left\{ 2^{^{\underline{2}}},5^{^{\underline{2}}}\right\} ,\left\{ 3^{^{\underline{2}}},8^{^{\underline{3}}}\right\} ,\left\{ 2^{^{\underline{3}}},3^{^{\underline{3}}}\right\} ,\left\{ 4^{^{\underline{3}}},7^{^{\underline{3}}}\right\} ,\left\{ 5^{^{\underline{3}}},6^{^{\underline{3}}}\right\} \right\} \\
\pi^{-1}\left(1\right) & = & \left\{ 5^{^{\underline{1}}},7^{^{\underline{1}}},3^{^{\underline{2}}},5^{^{\underline{2}}},3^{^{\underline{3}}},9^{^{\underline{3}}}\right\} \\
\pi^{-1}\left(2\right) & = & \left\{ 6^{^{\underline{1}}},8^{^{\underline{1}}},4^{^{\underline{2}}},2^{^{\underline{3}}},4^{^{\underline{3}}},8^{^{\underline{3}}}\right\} \\
\pi^{-1}\left(3\right) & = & \left\{ 1^{^{\underline{1}}},2^{^{\underline{1}}},3^{^{\underline{1}}},2^{^{\underline{2}}},6^{^{\underline{2}}},1^{^{\underline{3}}},6^{^{\underline{3}}}\right\} \\
\pi^{-1}\left(4\right) & = & \left\{ 4^{^{\underline{1}}},1^{^{\underline{2}}},5^{^{\underline{3}}},7^{^{\underline{3}}},10^{^{\underline{3}}}\right\} 
\end{eqnarray*}
Then, the labelled array representing $\left(\mu,\pi\right)$ is given
by \ref{fig:Labelled Array}.

\begin{figure}
\begin{centering}
\includegraphics{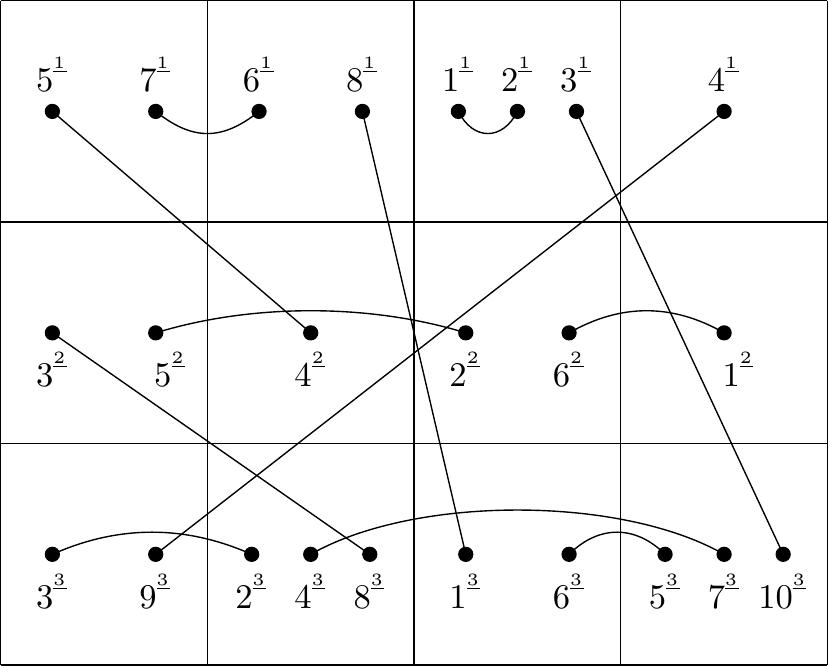}
\par\end{centering}
\caption{\label{fig:Labelled Array}A labelled array with 3 rows and 4 columns}
\end{figure}

Note that an $n\times K$ array with paired and labelled vertices
as described above uniquely represents a pairing $\mu\in\P_{n}^{\left(\mathbf{q};\mathbf{s}\right)}$
and a function $\pi\colon\left[p_{1},\dots,p_{n}\right]\rightarrow\left[K\right]$.
The condition $\pi\left(\mu\left(v\right)\right)=\pi\left(\gamma_{p_{1},p_{2},\dots,p_{n}}\left(v\right)\right)$
is fulfilled if and only if for every pair $\left\{ x^{\underline{i}},y^{\underline{k}}\right\} $
in the array, the vertex $\left(x+1\right)^{\underline{i}}$ is in
the same column as the vertex of $y^{\underline{k}}$, where the addition
$x+1$ is taken modulo $p_{i}$.

Next, we will show that this condition is sufficient to reconstruct
the array if the labels are removed and replaced by marked cells.
We do this by defining paired arrays as abstract combinatorial objects,
then creating a bijection between paired arrays and labelled arrays.
\begin{defn}
\label{def:Paired Array}Let $n,K\ge1$, $\mathbf{q}=\left(q_{1},\dots,q_{n}\right)\ge\mathbf{0}$,
$\mathbf{s}=\left(s_{1,2},s_{1,3},\dots,s_{n-1,n}\right)\ge\mathbf{0}$,
and $\mathbf{R}=\left(R_{1},\dots,R_{n}\right)\in\left[K\right]^{n}$.
We define $\PA{n}{K}{\mathbf{R}}{\mathbf{q}}{\mathbf{s}}$ to be the
set of \emph{paired arrays}, which are arrays of cells and vertices
subject to the following conditions.

\begin{itemize}
\item A paired array is an array of cells, arranged in $n$ rows and $K$
columns.
\item Each cell $\left(i,j\right)$ contains an ordered list of vertices,
arranged left to right, so that row $i$ contains $p_{i}\coloneqq2q_{i}+s_{i}=2q_{i}+\sum_{k<i}s_{k,i}+\sum_{k>i}s_{i,k}$
vertices in total.
\item Each vertex $u$ is paired with exactly one other vertex $v$, which
is called the \emph{partner} of $u$. Exactly $2q_{i}$ vertices of
row $i$ are paired with other vertices of row $i$, and for $i<k$,
exactly $s_{i,k}$ vertices of row $i$ are paired with vertices of
row $k$. Graphically, the pairings are denoted as edges between vertices.
\item Each row $i$ has exactly $R_{i}$ marked cells, which are denoted
by marking the cell with a box in its upper right corner.
\item A vertex $v$ is \emph{critical} if it is the rightmost vertex of
a cell, and the cell it belongs to is not marked. A pair $\left\{ u,v\right\} $
that contains a critical vertex is a \emph{critical pair}.
\item A pair of vertices $\left\{ u,v\right\} $ is a \emph{mixed pair}
if $u$ and $v$ belong to different rows. The vertices $u$ and $v$
are called \emph{mixed vertices}.
\item An \emph{object} of a paired array refers to either a vertex, or the
box used to indicate that a cell is marked. If a cell both contains
vertices and a box, the box is to be taken as the rightmost object
of the cell.
\end{itemize}
\end{defn}
Generally, we use $\alpha\in\PA{n}{K}{\mathbf{R}}{\mathbf{q}}{\mathbf{s}}$
to denote a paired array. Before introducing the conditions used in
Goulden and Slofstra, we will first introduce a number of useful notations
and conventions.
\begin{conv}
\label{conv:Array Convention}For notational convenience, we introduce
the following:

\begin{itemize}
\item We use calligraphic letters to denote columns or sets of columns.
For generic columns or sets of columns, we use the letters $\X$,
$\Y$, and $\Z$.
\item For each calligraphic letter, we use the corresponding upper case
letter to denote the number of columns in the set. For example, $X=\left|\X\right|$.
\item For each calligraphic letter, we use the corresponding lower case
letter, subscripted by the row number, to denote the total number
of vertices in those columns for a given row. For example, $x_{i}$
is the total number of vertices in row $i$ of the columns of $\X$.
\item We generally use $i,j,k,\ell$ as index variables, with $i$ and $k$
for rows, and $j$ and $\ell$ for columns. Furthermore, we use cell
$\left(i,j\right)$ to denote the cell in row $i$, column $j$ of
the array.
\item We use $\K$ to denote the set of all columns, and $K$ to denote
the total number of columns.
\item We use $\R_{i}$ to denote the set of columns that are marked in row
$i$, and $R_{i}$ to denote the number of columns that are marked
in row $i$.
\item We use $\F_{i}$ to denote the set of columns that have at least one
vertex in row $i$, and $F_{i}$ to denote the number of columns that
are marked in row $i$.
\item We use $w_{i,j}$ to denote the number of vertices in cell $\left(i,j\right)$,
and $\mathbf{w}$ to denote a matrix of $w_{i,j}$ describing the
number of vertices in each cell of row $i$.
\item We let $s_{i,k}=s_{k,i}$ for $i>k$, and $s_{i}=\sum_{k\neq i}s_{i,k}$
be the total number of mixed vertices of row $i$. This means that
row $i$ contains $p_{i}=2q_{i}+s_{i}$ vertices.
\end{itemize}
\end{conv}
With these conventions, we are ready to define the two conditions
that allow us to create a bijection between labelled arrays and paired
arrays.
\begin{defn}
\label{def:Paired Array Conditions}Let $\alpha\in\PA{n}{K}{\mathbf{R}}{\mathbf{q}}{\mathbf{s}}$
be a paired array.

\begin{itemize}
\item $\alpha$ is said to satisfy the \emph{balance condition} if for each
cell $\left(i,j\right)$, the number of mixed vertices in cell $\left(i,j\right)$
is equal to the number of mixed pairs $\left\{ u,v\right\} $ such
that $u$ is in row $i$ and $v$ is in column $j$ (but not row $i$).
\item For each row $i$, the \emph{forest condition function} $\psi_{i}\colon\F_{i}\backslash\R_{i}\mapsto\K$
is defined as follows: For each column $j\in\F_{i}\backslash\R_{i}$,
if the rightmost vertex $v$ is paired with a vertex $u$ in column
$\ell$, then $\psi_{i}\left(j\right)=\ell$. $\alpha$ is said to
satisfy the \emph{forest condition} if for each row $i$, the functional
digraph of $\psi_{i}$ on the vertex set $\F_{i}\cup\psi_{i}\left(\F_{i}\right)\cup\R_{i}$
is a forest with root vertices $\R_{i}$. That is, for each column
$j\in\F_{i}\backslash\R_{i}$, there exists some positive integer
$t$ such that $\psi_{i}^{t}\left(j\right)\in\R_{i}$. Note that we
always include $\R_{i}$ in the vertex set of the functional digraph
of $\psi_{i}$, regardless of whether they are in the domain or range
of $\psi_{i}$.
\end{itemize}
A paired array is \emph{proper} if it satisfies the balance and forest
conditions. A paired array is called a \emph{canonical array} if it
is proper and $\mathbf{R}=\mathbf{1}$. We denote the set of canonical
arrays as $\CA{n}{K}{\mathbf{q}}{\mathbf{s}}$, and we let $c_{n,K}^{\left(\mathbf{q};\mathbf{s}\right)}=\left|\CA{n}{K}{\mathbf{q}}{\mathbf{s}}\right|$.
A paired array is called a \emph{vertical array} if for every pair
$\left\{ u,v\right\} $, $u$ and $v$ are in different rows, and
is \emph{proper} if it satisfies the balance and forest conditions.
We denote the set of vertical arrays as $\VA{n}{K}{\mathbf{R}}{\mathbf{s}}=\PA{n}{K}{\mathbf{R}}{\mathbf{0}}{\mathbf{s}}$
and the set of proper vertical arrays as $\PVA{n}{K}{\mathbf{R}}{\mathbf{s}}$.
We also let $v_{n,K;\mathbf{R}}^{\left(\mathbf{s}\right)}=\left|\PVA{n}{K}{\mathbf{R}}{\mathbf{s}}\right|$.
For notational convenience, we extend our definition of $v_{n,K;\mathbf{R}}^{\left(\mathbf{s}\right)}$
to all $\mathbf{R}\ge\mathbf{1}$ by letting $v_{n,K;\mathbf{R}}^{\left(\mathbf{s}\right)}=0$
if $R_{i}>K$ for some $1\le i\le n$. Again, unlike in Goulden and
Slofstra, we do not have the non-empty condition in our definition
of the proper paired array.
\end{defn}
Note that we will generally not work directly with paired arrays that
do not satisfy the forest condition. However, as vertical arrays not
satisfying the forest condition are vital for extending paired arrays,
we have separated the forest condition from the definition of vertical
arrays itself. Of the two conditions in \ref{def:Paired Array Conditions},
the forest condition is more fundamental, and all the arrays we define
in this paper will satisfy some form of this condition. The balance
condition is in general difficult to handle, but can be radically
simplified if the support graph of $\mathbf{s}$ forms a tree. For
convenience, arrays that have such property are called \emph{tree-shaped}.
With tree-shaped arrays, we can reduce the balance condition to a
condition that only depends on the number of mixed vertices in a cell,
essentially allowing us to ignore it.
\begin{lem}
\label{lem:tree shape balance}Let $\alpha\in\PA{n}{K}{\mathbf{R}}{\mathbf{q}}{\mathbf{s}}$
be a tree-shaped paired array, and suppose that $s_{i,k,j}$ is the
number of vertices in cell $\left(i,j\right)$ that are paired with
a vertex in row $k$ for all $1\le i,k\le n$ and $1\le j\le K$.
Then, $\alpha$ satisfies the balance condition if and only if $s_{i,k,j}=s_{k,i,j}$
for all $i\neq k$.
\end{lem}
\begin{proof}
First, note that $s_{i,k}$ is the number of mixed pairs $\left\{ u,v\right\} $
with $u$ in row $i$ and $v$ in row $k$, so $s_{i,k}=\sum_{j}s_{i,k,j}$.
Also, let $x_{i,j}$ be the number of mixed vertices in cell $\left(i,j\right)$,
and observe that $x_{i,j}=\sum_{k\neq i}s_{i,k,j}$. Suppose $s_{i,k,j}=s_{k,i,j}$
for all $1\le i,k\le n$ and $1\le j\le K$. Then, by summing over
all $k\neq i$, we have $x_{i,j}=\sum_{k\neq i}s_{i,k,j}=\sum_{k\neq i}s_{k,i,j}$.
As $s_{k,i,j}$ is the number of mixed vertices in cell $\left(k,j\right)$
that are paired with a vertex in row $i$, the latter sum counts the
number of mixed pairs $\left\{ u,v\right\} $ such that $u$ is in
row $i$ and $v$ is in row $k$. Therefore, $\alpha$ satisfies the
balance condition.

Conversely, suppose $\alpha$ satisfies the balance condition. By
the same reasoning, we have $x_{i,j}=\sum_{k\neq i}s_{i,k,j}=\sum_{k\neq i}s_{k,i,j}$.
We will show by induction that $s_{i,k,,j}=s_{k,i,j}$ for all $i\neq k$.

Let $G$ be the support graph of $\mathbf{s}$ and suppose $G$ is
a tree. Without loss of generality, let the vertex $n$ be a leaf
of $G$, and assume that it is adjacent to the vertex $n-1$. As $n$
is not joined to other vertices in $G$, we have $s_{n,k,j}=s_{k,n,j}=0$
for all $1\le k\le n-2$ and $1\le j\le K$. Substituting this into
$\sum_{k\neq n}s_{n,k,j}=\sum_{k\neq i}s_{k,n,j}$, we obtain $s_{n,n-1,j}=s_{n-1,n,j}$.
This gives $s_{n,k,j}=s_{k,n,j}$ for $1\le k\le n-1$ and $1\le j\le K$.

Now, let $s_{i,k}^{\prime}=\sum_{j}s_{i,k,j}$ and $x_{i,j}^{\prime}=\sum_{k\neq i,n}s_{i,k,j}$
for $1\le i,k\le n-1$, $i\neq k$, and $1\le j\le K$. That is, we
have effectively removed the last row of $\alpha$. Then, 
\begin{eqnarray*}
\sum_{k\neq i,n}s_{i,k,j} & = & \sum_{k\neq i}s_{i,k,j}-s_{i,n,j}\\
 & = & \sum_{k\neq i}s_{k,i,j}-s_{n,i,j}\\
 & = & \sum_{k\neq i,n}s_{k,i,j}
\end{eqnarray*}
by using the fact that $s_{i,n,j}=s_{n,i,j}$, and substituting in
the identity for $x_{i,j}$. Furthermore, as $s_{i,k}^{\prime}=s_{i,k}$
for $1\le i,k\le n-1$, the support graph given by $\mathbf{s}^{\prime}$
is $G\backslash\left\{ n\right\} $. As $n$ is a leaf of $G$, $G\backslash\left\{ n\right\} $
is also a tree. By the inductive hypothesis, $s_{i,k,j}=s_{k,i,j}$
for all $1\le i,k\le n-1$ and $1\le j\le K$, where $i\neq k$.

Therefore, $\alpha$ satisfies the balance condition if and only if
$s_{i,k,j}=s_{k,i,j}$ for all $i\neq k$, as desired.
\end{proof}
Now that we have defined the necessary framework for paired arrays,
we will state the relationship between canonical arrays and labelled
arrays.
\begin{thm}
\label{thm:Label to Canonical Array}For $n,K\ge1$, $\mathbf{q}\ge\mathbf{0}$,
and $\mathbf{s}\ge\mathbf{0}$, we have $f_{n,K}^{\left(\mathbf{q};\mathbf{s}\right)}=c_{n,K}^{\left(\mathbf{q};\mathbf{s}\right)}$.
\end{thm}
The proof is essentially the same as that in Goulden and Slofstra,
but without the non-empty condition. To obtain the canonical array
from labelled array, we simply mark the cell that contains $1^{^{\left(i\right)}}$
in each row $i$, then remove the labels. To reconstruct the labelled
array and prove that it is a bijection, we use the same \emph{label
recovery procedure} introduced in their paper. As an example of this
bijection, we have transformed the labelled array depicted in \ref{fig:Labelled Array}
into the canonical array depicted in \ref{fig:Canonical Array}.

\begin{figure}
\begin{centering}
\includegraphics{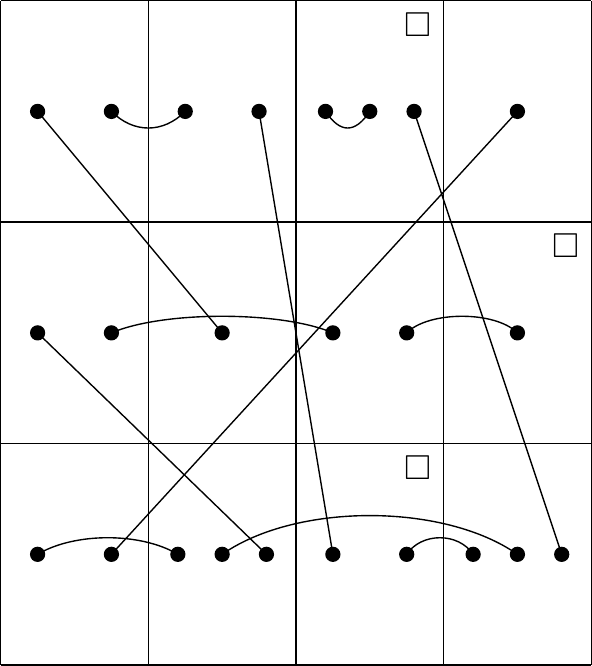}
\par\end{centering}
\caption{\label{fig:Canonical Array}A canonical paired array with 3 rows and
4 columns}
\end{figure}

Now that we know that canonical arrays are in bijection with labelled
arrays with the same parameters, the problem of enumerating maps on
surfaces reduces to that of enumerating canonical arrays. To solve
the latter problem, we will extend the procedure in Goulden and Slofstra
to remove all non-mixed pairs. Then, we will decompose the resulting
paired arrays via induction, removing one row at a time.
\begin{thm}
\label{thm:Canonical Vertical Formula}Let $n,K\ge1$, $\mathbf{q}\ge\mathbf{0}$,
and $\mathbf{s}\ge\mathbf{0}$. We have
\[
c_{n,K}^{\left(\mathbf{q};\mathbf{s}\right)}=\sum_{\mathbf{t}=\mathbf{0}}^{\mathbf{q}}\prod_{i=1}^{n}\frac{\left(2q_{i}+s_{i}\right)!}{2^{t_{i}}t_{i}!\left(s_{i}+q_{i}-t_{i}\right)!}\cdot v_{n,K;\mathbf{q-t+1}}^{\left(\mathbf{s}\right)}
\]
Furthermore, if $v_{n,K;\mathbf{R}}^{\left(\mathbf{s}\right)}$ can
be written as a polynomial expression in $K$ for all $R_{i}$, where
$1\le R_{i}\le q_{i}+1$, then $c_{n,K}^{\left(\mathbf{q};\mathbf{s}\right)}$
can be written as a polynomial expression in $K$.
\end{thm}
\begin{proof}
Despite not having the non-empty condition in our definition of the
paired functions and paired array, the proof of this theorem is essentially
the same as that of Goulden and Slofstra. The polynomiality of $c_{n,K}^{\left(\mathbf{q};\mathbf{s}\right)}$
follows from the fact that the summation bounds are independent of
$K$, so $c_{n,K}^{\left(\mathbf{q};\mathbf{s}\right)}$ as expressed
above is a polynomial combination of $v_{n,K;\mathbf{q-t+1}}^{\left(\mathbf{s}\right)}$,
with coefficients that are also independent of $K$.
\end{proof}
For example, by decomposing the canonical array in \ref{fig:Canonical Array},
we can obtain the vertical array in \ref{fig:Vertical Array}. Then,
by combining the theorems we have so far, we can write the generating
series in terms of the number of vertical arrays.
\begin{cor}
Let $n,K\ge1$, $\mathbf{q}\ge\mathbf{0}$, and $\mathbf{s}\ge\mathbf{0}$.
We have
\begin{eqnarray*}
A_{n}^{\left(\mathbf{q};\mathbf{s}\right)}\left(K\right) & = & \sum_{\mathbf{t}=\mathbf{0}}^{\mathbf{q}}\prod_{i=1}^{n}\frac{\left(2q_{i}+s_{i}\right)!}{2^{t_{i}}t_{i}!\left(s_{i}+q_{i}-t_{i}\right)!}\cdot v_{n,K;\mathbf{q-t+1}}^{\left(\mathbf{s}\right)}
\end{eqnarray*}
\end{cor}
\begin{proof}
By combining \ref{eq:Genfunc Label Relation}, \ref{thm:Label to Canonical Array},
and \ref{thm:Canonical Vertical Formula}, the result immediately
follows.
\end{proof}
\begin{figure}
\begin{centering}
\includegraphics{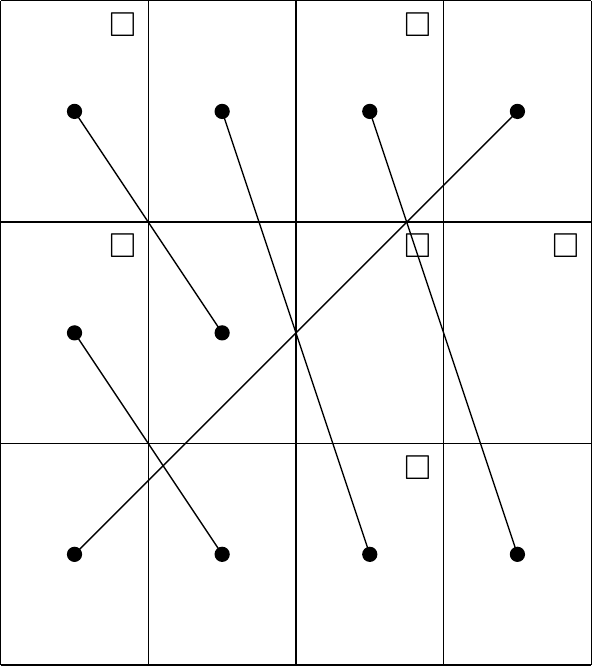}
\par\end{centering}
\caption{\label{fig:Vertical Array}Proper vertical array from the decomposition
of \ref{fig:Canonical Array}}
\end{figure}

\section{\label{sec:Arrowed Array Definitions}Definitions and Terminology
of Arrowed Arrays}

In this section, we will extend two-row paired arrays by the addition
of arrows, which represent hypothetical critical vertices. This will
allow us recursively decompose vertical arrays into arrowed arrays
and smaller vertical arrays. Some of the definitions and theorems
are taken directly from \cite{Chan:2017-1}, while others are direct
extensions. For the sake of length, we will omit the proofs of those
theorems.
\begin{defn}
\label{def:Arrowed Array}Let $K\ge1$, $s\ge0$, and $1\le R_{1},R_{2}\le K$.
An \emph{arrowed array} is a pair $\left(\alpha,\phi\right)$, where
$\alpha\in\VA{2}{K}{R_{1},R_{2}}{s}$ is a two-row vertical array,
and $\phi\colon\K\backslash\R_{1}\rightarrow\K$ is a partial function
from $\H\subseteq\K\backslash\R_{1}$ to $\K$, with $\R_{1}$ being
the set of marked columns in row 1 of $\alpha$. Graphically, $\phi$
is denoted by arrows drawn above row 1, where an arrow from $j$ to
$j^{\prime}$ is drawn if $j\in\H$ and $\phi\left(j\right)=j^{\prime}$.
For convenience, the two ends of the arrow belonging to columns $j$
and $j^{\prime}$ are called the \emph{arrow-tail} and \emph{arrow-head}
respectively, and column $j$ is said to \emph{point to} column $j^{\prime}$.
Furthermore, both the arrow-tail and arrow-head belong to row 1 of
their respective columns.

With the generalization of paired arrays to arrowed arrays, there
are corresponding generalizations of the terms and conventions used
to describe paired arrays. These generalizations will be compatible
with the conventions for paired arrays if the partial function $\phi$
is empty.

\begin{itemize}
\item An \emph{object} of $\left(\alpha,\phi\right)$ refers to either a
vertex, a box, or an arrow-tail. If a cell both contains vertices
and a box, or vertices and an arrow-tail, either the box or the arrow-tail
is to be taken as the rightmost object of the cell.
\item A vertex $v$ of an arrowed array is \emph{critical }if it is the
rightmost vertex of a cell, and the cell it belongs to is neither
marked nor contains an arrow-tail.
\item $\left(\alpha,\phi\right)$ is said to satisfy the \emph{non-empty
condition} if for each column $j$, there exists at least one cell
that contains an object.
\item $\left(\alpha,\phi\right)$ is said to satisfy the \emph{balance condition}
if for each column $j$, the number of vertices in cell $\left(1,j\right)$
is equal to the number of vertices in cell $\left(2,j\right)$.
\item Let $\F_{i}$ be the set of columns in row $i$ that contain at least
one vertex. The \emph{forest condition function} $\psi_{1}\colon\left(\H\cup\F_{1}\right)\backslash\R_{1}\mapsto\K$
for row 1 is defined as follows: For each column $j\in\H$, let $\psi_{1}\left(j\right)=\phi\left(j\right)$;
for $j\in\F_{1}\backslash\left(\H\cup\R_{1}\right)$, if the rightmost
vertex $v$ is paired with a vertex $u$ in column $j^{\prime}$,
let $\psi_{1}\left(j\right)=j^{\prime}$. The forest condition function
$\psi_{2}$ for row 2 is defined to be the same as the one for paired
arrays in \ref{def:Paired Array Conditions}. $\left(\alpha,\phi\right)$
is said to satisfy the \emph{forest condition} if the functional digraph
of $\psi_{1}$ on the vertex set $\H\cup\F_{1}\cup\psi_{1}\left(\H\cup\F_{1}\right)\cup\R_{1}$
is a forest with root vertices $\R_{1}$, and the functional digraph
of $\psi_{2}$ on the vertex set $\F_{2}\cup\psi_{2}\left(\F_{2}\right)\cup\R_{2}$
is a forest with root vertices $\R_{2}$. That is, for each column
$j\in\left(\H\cup\F_{1}\right)\backslash\R_{1}$, there exists some
positive integer $t$ such that $\psi_{1}^{t}\left(j\right)\in\R_{1}$,
and for each column $j\in\F_{2}\backslash\R_{2}$, there exists some
positive integer $t$ such that $\psi_{2}^{t}\left(j\right)\in\R_{2}$.
\item Additionally, $\left(\alpha,\phi\right)$ is said to satisfy the \emph{full
condition} if every cell contains at least one object.
\end{itemize}
The set of arrowed arrays that satisfies the forest condition is denoted
$\AR{K}{R_{1}}{R_{2}}{s}$.
\end{defn}
Notice in particular that a cell cannot contain both an arrow-tail
and be marked at the same time. Unless otherwise stated, we will continue
to use the conventions for paired arrays defined in \ref{conv:Array Convention}
for arrowed arrays. However, we will be using the definition of critical
vertex defined here instead of the one in \ref{def:Paired Array}.
As with paired arrays, we will always include the columns $\R_{i}$
in the vertex set for the functional digraph of $\psi_{i}$, regardless
of whether they are in the range of $\psi_{i}$. Note that permuting
the columns of an arrowed array does not change whether the array
satisfies the balance or forest conditions, as all this action does
is to relabel the vertices of the functional digraph. Furthermore,
to reduce cluttering, we will draw the boxes for row 2 at the lower
right corner instead of the upper right. An example of an arrowed
array that satisfies the forest condition can be found in \ref{fig:Arrowed Array}.

\begin{figure}
\begin{centering}
\includegraphics{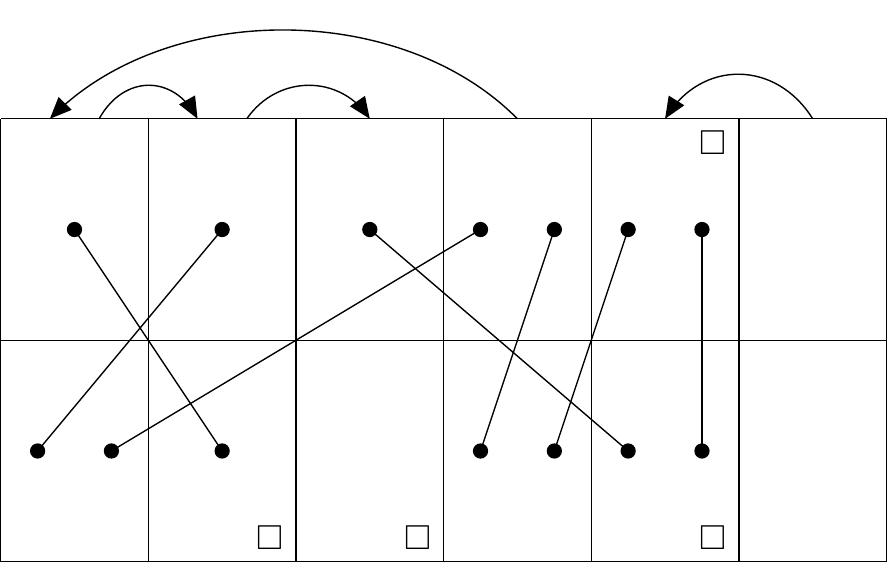}
\par\end{centering}
\caption{\label{fig:Arrowed Array}A arrowed array in $\AR{6}{1}{3}{7}$}
\end{figure}

While the parameters used for defining the set of arrowed arrays is
natural with respect to paired arrays, it does not easily lend itself
to a formula. To make it manageable for summation, we need to partition
the set of arrowed arrays by adding further constraints, which will
take for form of three different substructures.
\begin{defn}
\label{def:Substructure}Let $K\ge1$, $s\ge0$, and $1\le R_{1},R_{2}\le K$.
A \emph{substructure} $\Theta$ of $\AR{K}{R_{1}}{R_{2}}{s}$ is a
set of constraints that defines a subset of $\AR{K}{R_{1}}{R_{2}}{s}$.
For convenience, an arrowed array $\left(\alpha,\phi\right)$ is said
to satisfy $\Theta$ if $\left(\alpha,\phi\right)$ satisfies the
constraints given by $\Theta$. In particular, here are the three
substructures that we will use in this paper.

\begin{itemize}
\item Let $\mathbf{w}$ be a non-negative matrix of size $2\times K$, $\R_{1}$
and $\R_{2}$ be $R_{1}$ and $R_{2}$ subsets of $\K$, and $\phi$
be a partial function from $\H\subseteq\K\backslash\R_{1}$ to $\K$.
The substructure $\Gamma=\left(\mathbf{w},\R_{1},\R_{2},\phi\right)$
is defined to be the subset of $\AR{K}{R_{1}}{R_{2}}{s}$, such that
for each pair $\left(\alpha^{\prime},\phi^{\prime}\right)\in\AR{K}{R_{1}}{R_{2}}{s}$,
$\alpha^{\prime}$ contains $w_{i,j}$ vertices in cell $\left(i,j\right)$,
the marked cells in row 1 and 2 of $\alpha^{\prime}$ are $\R_{1}$
and $\R_{2}$, respectively, and $\phi^{\prime}=\phi$.
\item Let $\mathbf{w}$ be a non-negative vector of size $K$, $\R_{1}$
be an $R_{1}$ subset of $\K$, and $\phi$ be a partial function
from $\H\subseteq\K\backslash\R_{1}$ to $\K$. The substructure $\Delta=\left(\mathbf{w},\R_{1},\phi\right)$
is defined to be the subset of $\AR{K}{R_{1}}{R_{2}}{s}$, such that
for each pair $\left(\alpha^{\prime},\phi^{\prime}\right)\in\AR{K}{R_{1}}{R_{2}}{s}$,
$\left(\alpha^{\prime},\phi^{\prime}\right)$ satisfies the balance
condition, $\alpha^{\prime}$ contains $w_{j}$ vertices in both cells
$\left(1,j\right)$ and $\left(2,j\right)$, the marked cells in row
1 of $\alpha^{\prime}$ is $\R_{1}$, and $\phi^{\prime}=\phi$. Furthermore,
for $A\ge0$, we define $\Delta_{A}$ to be the substructure that
describes the subset of arrowed arrays that satisfies $\Delta$, and
have exactly $A$ columns of type $\A$.
\item Let $\P$ be a subset of $\K$ with $\left|\P\right|\ge R_{1}\ge1$,
$\mathbf{x}$ be a non-negative vector of size $K$, and $\phi\colon\K\backslash\P\rightarrow\K$
be a partial function from $\H\subseteq\K\backslash\P$ to $\H\cup\P$.
Suppose that $x_{j}=0$ for all $j\notin\H\cup\P$ and $s$ be such
that $\sum_{j}x_{j}=s-\left|\P\right|+R_{1}$. The substructure $\Lambda=\left(\mathbf{x},\P,\phi\right)$
is defined to be the subset of $\AR{K}{R_{1}}{R_{2}}{s}$, such that
for each pair $\left(\alpha^{\prime},\phi^{\prime}\right)\in\AR{K}{R_{1}}{R_{2}}{s}$,
$\left(\alpha^{\prime},\phi^{\prime}\right)$ satisfies the balance
condition, the set of marked cells in row 1 of $\alpha^{\prime}$
is a subset of $\P$, and $\phi^{\prime}=\phi$. Furthermore, for
each column $j\in\H\cup\P$, both cells $\left(1,j\right)$ and $\left(2,j\right)$
contains $x_{j}+1$ vertices if $j\in\P$ and is unmarked, and $x_{j}$
vertices otherwise.
\end{itemize}
For convenience, we say a substructure $\Theta$ is a \emph{refinement}
of another substructure $\Theta^{\prime}$ if the set of arrowed arrays
satisfying $\Theta$ is a subset of the arrowed arrays satisfying
$\Theta^{\prime}$. We denote it as $\Theta\hookrightarrow\Theta^{\prime}$.
Furthermore, if $\Theta_{1},\dots,\Theta_{t}$ is a set of substructures
that are refinements of a substructure $\Theta^{\prime}$, we say
that $\Theta_{1},\dots,\Theta_{t}$ \emph{partitions} $\Theta^{\prime}$
if the sets of arrowed arrays satisfying the $\Theta_{i}$'s are mutually
disjoint, and their union is the set of arrowed arrays that satisfy
$\Theta$. Finally, we will use arrowed array terminologies such as
critical vertex, arrow-head, and points to with substructures when
they are applicable.
\end{defn}
Note that the latter substructures in \ref{def:Substructure} can
be partitioned using the substructure directly above. Furthermore,
the substructures $\Delta_{A}$ are refinements of substructures $\Delta$.
Also, for substructure $\Lambda$, the vertices are restricted to
the columns $\H\cup\P$, and $\mathbf{x}$ represents the number of
non-critical vertices in row 1.
\begin{lem}
\label{lemma:Arrow Simplification Substructure-1}Let $\Gamma=\left(\mathbf{w},\R_{1},\R_{2},\phi\right)$
be a substructure of $\AR{K}{R_{1}}{R_{2}}{s}$, and suppose that
$\phi$ contains a column $\X$ that points to a column $\Y$, with
cell $\left(1,\Y\right)$ marked. Let $\Gamma^{\prime}=\left(\mathbf{w},\R_{1}\cup\left\{ \X\right\} ,\R_{2},\phi^{\prime}\right)$
be a substructure of $\AR{K}{R_{1}+1}{R_{2}}{s}$, such that
\begin{eqnarray*}
\phi^{\prime}\left(j\right) & = & \begin{cases}
\mbox{undefined} & j=\X\\
\phi\left(j\right) & j\in\H\backslash\X\mbox{ },
\end{cases}
\end{eqnarray*}
that is, instead of pointing to $\Y$, we mark cell $\left(1,\X\right)$
of $\Gamma^{\prime}$. Then, the number of arrowed arrays satisfying
$\Gamma$ and the number of arrowed arrays satisfying $\Gamma^{\prime}$
are equal. Furthermore, $\Gamma$ satisfies the balance, non-empty,
and full conditions if and only if $\Gamma^{\prime}$ satisfies them,
respectively.
\end{lem}
The proof of this lemma can be found in \cite{Chan:2017-1}, and by
changing the proof slightly, we can show that for substructure $\Delta$,
the number of arrowed arrays satisfying $\Delta=\left(\mathbf{w},\R_{1},\phi\right)$
and $\Delta^{\prime}=\left(\mathbf{w},\R_{1}\cup\left\{ \X\right\} ,\phi^{\prime}\right)$
are equal.
\begin{lem}
\label{lemma:Arrow Simplification Substructure-2}Let $\Gamma=\left(\mathbf{w},\R_{1},\R_{2},\phi\right)$
be a substructure of $\AR{K}{R_{1}}{R_{2}}{s}$, and suppose that
$\phi$ contains a column $\X$ that points to a column $\Y$, and
the column $\Y$ points to another column $\Z$. Let $\Gamma^{\prime}=\left(\mathbf{w},\R_{1},\R_{2},\phi^{\prime}\right)$
be a substructure of $\AR{K}{R_{1}}{R_{2}}{s}$ such that
\begin{eqnarray*}
\phi^{\prime}\left(j\right) & = & \begin{cases}
\Z & j=\X\\
\phi\left(j\right) & j\in\H\backslash\X\mbox{ },
\end{cases}
\end{eqnarray*}
that is, instead of pointing to $\Y$, $\X$ now points to $\Z$ in
$\phi^{\prime}$. Then, the number of arrowed arrays satisfying $\Gamma$
and the number of arrowed arrays satisfying $\Gamma^{\prime}$ are
equal. Furthermore, $\Gamma$ satisfies the balance, non-empty, and
full conditions if and only if $\Gamma^{\prime}$ satisfies them,
respectively.
\end{lem}
Similarly, the proof of this lemma can be adapted to show that number
of arrowed arrays satisfying $\Delta=\left(\mathbf{w},\R_{1},\phi\right)$
and $\Delta^{\prime}=\left(\mathbf{w},\R_{1},\phi^{\prime}\right)$
are equal, and the same for the number of arrowed arrays satisfying
$\Lambda=\left(\mathbf{x},\P,\phi\right)$ and $\Lambda^{\prime}=\left(\mathbf{y},\P,\phi^{\prime}\right)$.

Collectively, these are the \emph{arrow simplification lemmas}, and
pictures describing the applications of these lemmas can be found
in \ref{fig:Arrow Simplification-1} and \ref{fig:Arrow Simplification-2}.
Furthermore, applying these lemmas to the array in \ref{fig:Arrowed Array}
gives us \ref{fig:Irreducible Arrowed Array}. Note that these lemmas
can be applied repeatedly to simplify a substructure, until either
all arrow-heads are in cells that are unmarked and have no arrow-tails,
or an arrow-head is in the same cell as its own arrow-tail. We are
only interested in the former, as the latter implies that there is
a cycle in the functional digraph of $\phi$, which violates the forest
condition. This gives rise to the following definition.
\begin{defn}
\label{def:Irreducible Substructure Lambda}A substructure $\Theta$
is \emph{irreducible} if the functional digraph of $\phi$ is acyclic,
and $\Theta$ cannot be further simplified with the application of
arrow simplification lemmas. Any cell of an irreducible substructure
containing an arrow-head must be unmarked in row 1, and cannot contain
an arrow-tail. Furthermore, it follows from definition that if an
irreducible substructure satisfies the full condition, then any cell
containing an arrow-head must also contain a critical vertex in row
1.
\end{defn}
Note that for substructure $\Lambda$, only the second arrow simplification
lemma applies. Furthermore, for substructure $\Lambda=\left(\mathbf{x},\P,\phi\right)$,
the arrow-heads must be in cells of $\H\cup\P$, but they cannot be
in $\H$ for $\Lambda$ to be irreducible. Hence, if $\Lambda$ is
irreducible, then $\phi$ must be a function from $\H$ to $\P$.

\begin{figure}
\begin{centering}
\includegraphics{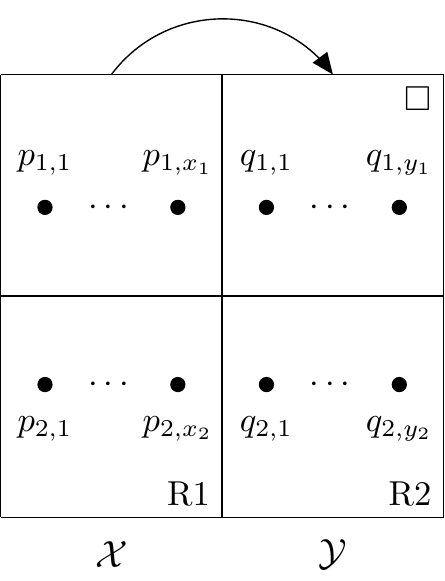}$\qquad\qquad\qquad$\includegraphics{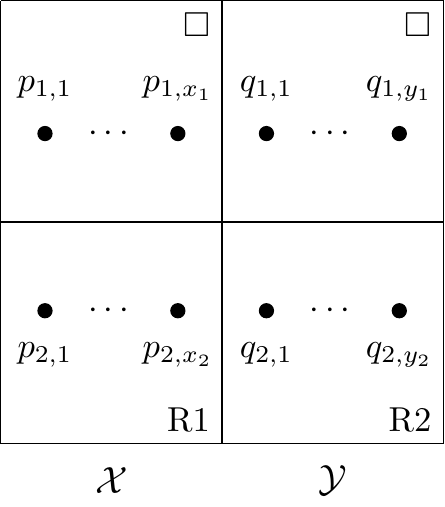}
\par\end{centering}
By applying the arrow simplification procedure to the left figure,
we arrive at the right figure. R1 and R2 can be arbitrary in whether
they are marked, but they must be the same between the two figures.

\caption{\label{fig:Arrow Simplification-1}Arrow Simplification 1}
\end{figure}

\begin{figure}
\begin{centering}
\includegraphics{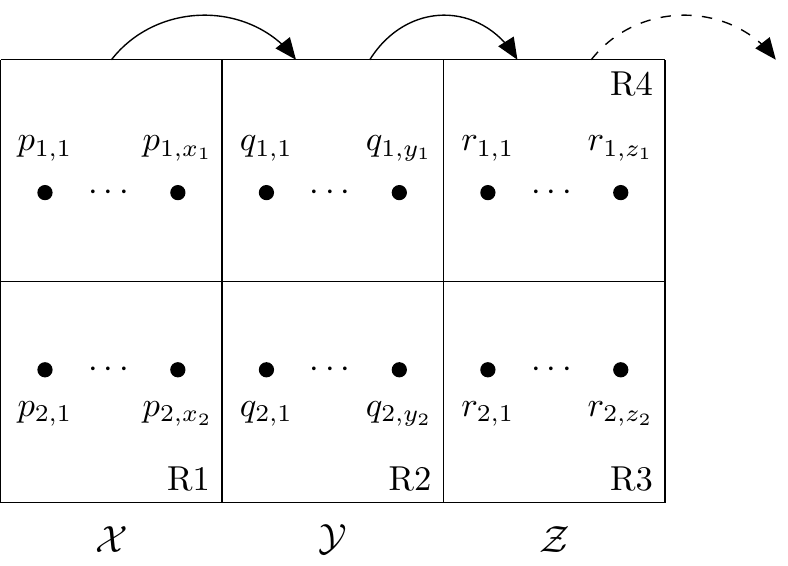}
\par\end{centering}
\begin{centering}
\includegraphics{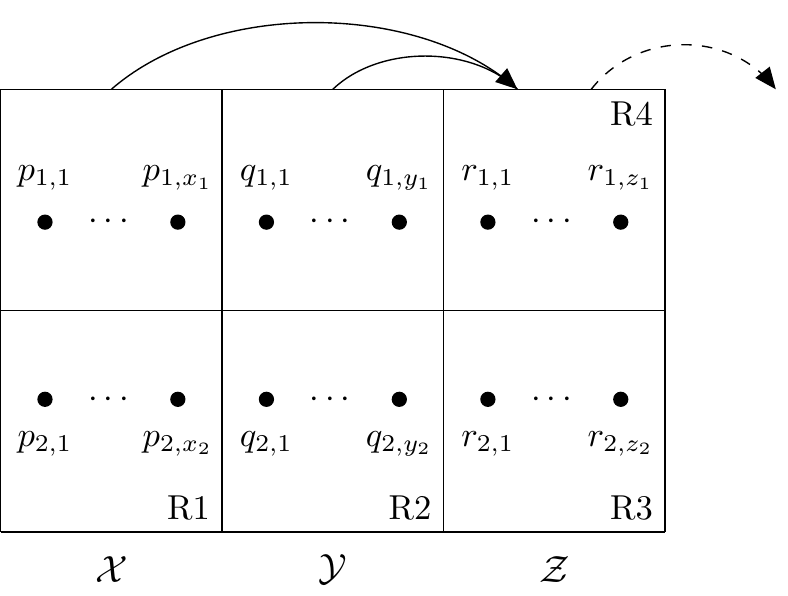}
\par\end{centering}
By applying the arrow simplification procedure to the top figure,
we arrive at the bottom figure. R1, R2, R3, and R4 can be arbitrary
in whether they are marked, but they must be the same between the
two figures. The same holds for the optional arrow with $\Z$ as its
tail.

\caption{\label{fig:Arrow Simplification-2}Arrow Simplification 2}
\end{figure}

\begin{figure}
\begin{centering}
\includegraphics{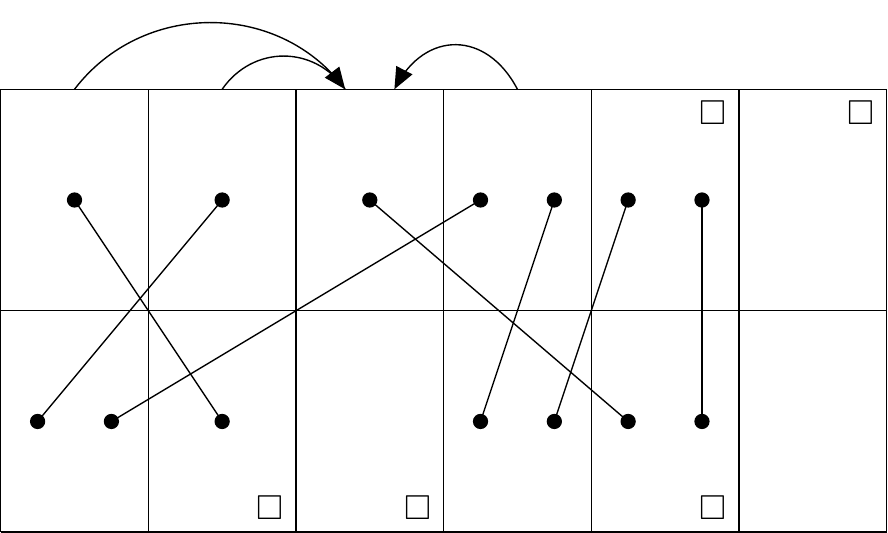}
\par\end{centering}
\caption{\label{fig:Irreducible Arrowed Array}Simplification of the arrowed
array in \ref{fig:Arrowed Array} into an irreducible array}
\end{figure}

\begin{defn}
\label{def:Column type definition}If $\Gamma=\left(\mathbf{w},\R_{1},\R_{2},\phi\right)$
is an irreducible substructure, then we can categorize the columns
of $\Gamma$ as follows: Let $\A,\B,\C,\D$ be a partition of the
columns of $\K\backslash\H$, where

\begin{itemize}
\item Columns in $\A$ have both row 1 and row 2 unmarked
\item Columns in $\B$ have row 1 marked and row 2 unmarked
\item Columns in $\C$ have row 1 unmarked and row 2 marked
\item Columns in $\D$ have both row 1 and row 2 marked
\end{itemize}
Furthermore, if $\X$ is a column or a set of columns, let $\XB$
and $\XP$ be the sets of columns that have arrows pointing to $\X$,
and that have row 2 unmarked and marked, respectively. In particular,
$\AB$ and $\AP$ denotes the sets of columns pointing to $\A$, and
$\CB$ and $\CP$ denotes the sets of columns pointing to $\C$, with
row 2 unmarked and marked, respectively. These sets of columns implicitly
defined by $\Gamma$ are referred to as \emph{column types}, and a
diagram with all the column types can be found in \ref{fig:Column Types}.
\end{defn}
\begin{figure}
\begin{centering}
\includegraphics{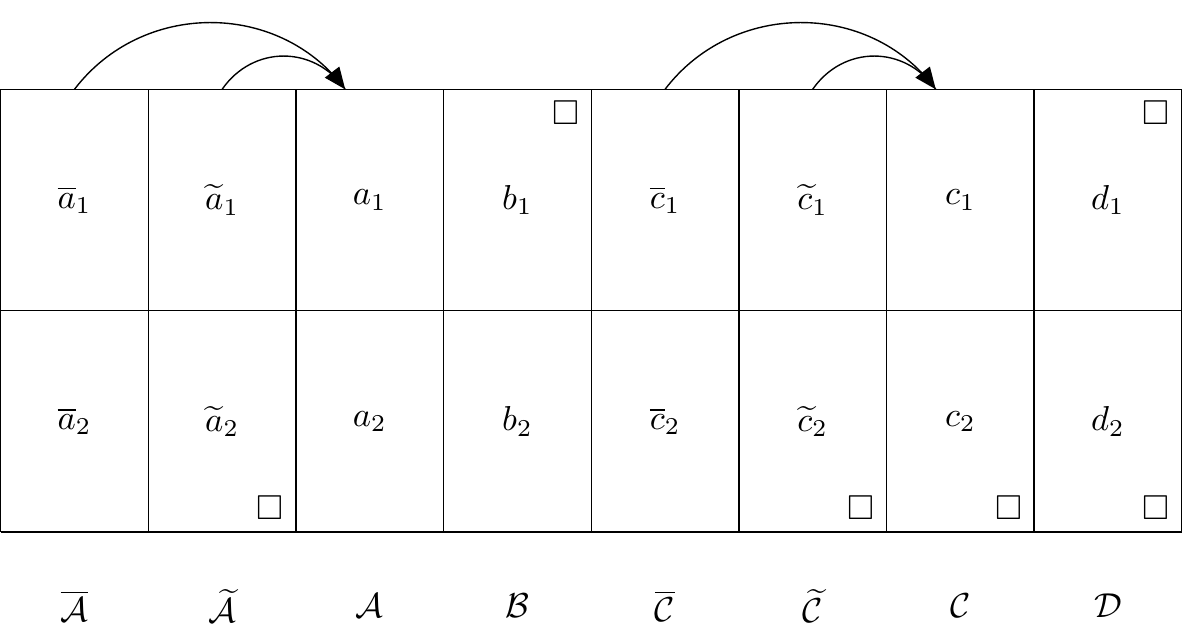}
\par\end{centering}
\caption{\label{fig:Column Types}Column types and variables for the number
of vertices}
\end{figure}

These eight column types form a partition of $\K$ on irreducible
substructures, and knowing the number of columns and the number of
vertices for each column type of $\Gamma$ is sufficient to count
the number of arrowed arrays satisfying it.

\section{\label{sec:Substructure Gamma Formula}Enumeration of Substructures}

Now, we have everything we need to provide formulas for the number
of arrowed arrays satisfying the substructures defined in \ref{def:Substructure}.
The first formula is proved in \cite{Chan:2017-1}, and enumerates
arrays satisfying substructure $\Gamma$.
\begin{thm}
\label{thm:Substructure general formula}Given an irreducible substructure
$\Gamma=\left(\mathbf{w},\R_{1},\R_{2},\phi\right)$ that satisfies
the full condition with $s\ge A+2$, the number of arrowed arrays
$\left(\alpha,\phi\right)\in\AR{K}{R_{1}}{R_{2}}{s}$ that satisfy
$\Gamma$ is given by the formula
\[
T\left(\Gamma\right)=\left(s-1\right)!\left[\frac{\left(b_{2}+d_{2}\right)\left(\ap_{1}+c_{1}+\cp_{1}+d_{1}\right)}{s-A}+\frac{b_{1}\left(c_{2}+\cb_{2}+\cp_{2}\right)-\cb_{1}\left(b_{2}+d_{2}\right)}{\left(s-A\right)\left(s-A-1\right)}\right]
\]
In the case where $s=A+1$, the formula reduces to 
\[
T\left(\Gamma\right)=\left(s-1\right)!\left(b_{2}+d_{2}\right)\left(\ap_{1}+c_{1}+\cp_{1}+d_{1}\right)
\]
\end{thm}
By the convention set out in \ref{conv:Array Convention}, we let
a lower case variable $x_{i}$ represent the total number of points
in row $i$ of the columns of type $\X$, and $A$ represent the number
of columns of type $\A$. For convenience, we will drop the subscripts
of the formula from here on, as we will only deal with arrays that
satisfy the balance condition.

Next, we provide a formula for substructure $\Delta$. This substructure
allows us to mark the cells of row 2 arbitrarily, while keeping the
positions of the marked cells in row 1 and the vertices fixed.
\begin{thm}
\label{thm:Substructure Delta General Formula}Let $R_{1},R_{2}\ge1$,
and let $\Delta=\left(\mathbf{w},\R_{1},\phi\right)$ be an irreducible
substructure that satisfies the balance condition. Furthermore, suppose
$w_{j}>0$ for $1\le j\le K$. Then, the number of arrowed arrays
$\left(\alpha,\phi\right)\in\AR{K}{R_{1}}{R_{2}}{s}$ with substructure
$\Delta$ is given by the formula
\[
T\left(\Delta\right)=s!\sum_{A=0}^{s-1}\frac{r}{s-A}\binom{M}{M-A}\binom{K-M-1}{R_{2}-M+A-1}
\]
where $r$ is the total number of vertices in row 1 of the columns
of $\R_{1}$, and $M$ is the number of columns that contain a critical
vertex in row 1.
\end{thm}
\begin{proof}
To prove this theorem, we sum $T\left(\Gamma\right)$ over all substructures
$\Gamma=\left(\mathbf{w},\R_{1},\R_{2},\phi\right)$ that are refinements
of $\Delta$. Since $\Delta$ satisfies the balance conditions and
$w_{j}>0$ for $1\le j\le K$, all substructures $\Gamma$ satisfy
the full condition, so we can use the formula of $T\left(\Gamma\right)$
given by \ref{thm:Substructure general formula}. Note that $T\left(\Gamma\right)$
only depends on the number of columns of type $\A$, even though it
depends on the number of vertices of other column types. Therefore,
we first sum over all $\Gamma$ with $A$ columns of type $\A$ to
obtain $T\left(\Delta_{A}\right)$, then we sum $A$ from 0 to $s-1$
to obtain $T\left(\Delta\right)$. As $\Delta$ satisfies the balance
condition, so must all $\Gamma$ that are refinements of $\Delta$.
This implies that we can drop the subscripts from $T\left(\Gamma\right)$.

Let $\M$ be the set of columns that contains a critical vertex in
row 1, and $\H$ be the set of columns that contains an arrow-tail.
Then, $\R_{1}$, $\M$, and $\H$ partitions $\K$. As $R_{1}\ge1$,
we have $M<K$. In the case where $M=0$, we have $\M=\H=\emptyset$
and $\R_{1}=\K$. Therefore, by simplifying and substituting in the
formula for $T\left(\Gamma\right)$, we have 
\[
T\left(\Delta\right)=\sum_{\Gamma\hookrightarrow\Delta}d\left(s-1\right)!
\]
Note that a vertex $v$ in cell $\left(1,\X\right)$ contributes to
$d$ if $\X$ is marked in row 2. As there are $\binom{K-1}{R_{2}-1}$
ways to mark the columns of $\K$ in row 2 with $\X$ marked, and
$s$ vertices in row 1, we have
\begin{eqnarray*}
T\left(\Delta\right) & = & s!\binom{K-1}{R_{2}-1}
\end{eqnarray*}
This result agrees with substituting $M=A=0$ into the formula for
$T\left(\Delta\right)$.

In the case where $1\le M\le K-1$, we have $\R_{1}=\B\cup\D$. This
gives us $r=b+d$, and allows us to rewrite $T\left(\Gamma\right)$
as 
\begin{eqnarray*}
T\left(\Gamma\right) & = & \left(s-1\right)!\left(T_{1}\left(\Gamma\right)+T_{2}\left(\Gamma\right)+T_{3}\left(\Gamma\right)+T_{4}\left(\Gamma\right)\right)
\end{eqnarray*}
where 
\begin{align*}
T_{1}\left(\Gamma\right) & =\frac{rc}{s-A} & T_{2}\left(\Gamma\right) & =\frac{r\left(\ap+\cp+d\right)}{s-A}\\
T_{3}\left(\Gamma\right) & =\frac{b\left(c+\cb+\cp\right)}{\left(s-A\right)\left(s-A-1\right)} & T_{4}\left(\Gamma\right) & =-\frac{r\cb}{\left(s-A\right)\left(s-A-1\right)}
\end{align*}
for $0\le A\le s-2$, with $T_{3}\left(\Gamma\right)=T_{4}\left(\Gamma\right)=0$
for $s=A-1$. As the substructures $\Gamma=\left(\mathbf{w},\R_{1},\R_{2},\phi\right)$
with $A$ columns of type $\A$ partitions $\Delta_{A}$, we can let
$T_{i}\left(\Delta_{A}\right)=\sum_{\Gamma\hookrightarrow\Delta_{A}}T_{i}\left(\Gamma\right)$
for $i=1,2,3,4$, which gives us 
\[
T\left(\Delta\right)=\left(s-1\right)!\left(\sum_{A=0}^{s-1}\left(T_{1}\left(\Delta_{A}\right)+T_{2}\left(\Delta_{A}\right)\right)+\sum_{A=0}^{s-2}\left(T_{3}\left(\Delta_{A}\right)+T_{4}\left(\Delta_{A}\right)\right)\right)
\]

To evaluate each of the $T_{i}\left(\Delta_{A}\right)$, we look at
the number of substructures $\Gamma$ such that a vertex or a pair
of vertices contributes to the numerator of $T_{i}\left(\Delta_{A}\right)$.
Note that we can ignore $r$ since it is the number of vertices in
$\R_{1}$, which is a constant with respect to $\Delta$. Of the three
sets of columns, only the columns of $\M$ can become columns of type
$\A$. Therefore, if a substructure $\Gamma$ is a refinement of $\Delta_{A}$,
it must have exactly $M-A$ marked cells in row 2 of $\M$. It must
also have exactly $R_{2}-M+A$ marked cells in row 2 of $\R_{1}\cup\H$.
This means in total, there are $\binom{M}{M-A}\binom{K-M}{R_{2}-M+A}$
substructures of the form $\Gamma=\left(\mathbf{w},\R_{1},\R_{2},\phi\right)$
that are refinements of $\Delta_{A}$.

Now, a vertex $v$ in row 1 of a column $\X$ contributes to $c$
if $\X\in\M$ and $\X$ is marked in row 2. As there are $\binom{M-1}{M-A-1}$
ways to mark the columns of $\M$ in row 2 with $\X$ marked, and
$\binom{K-M}{R_{2}-M+A}$ ways to mark the columns of $\K\backslash\M$,
$v$ contributes $\binom{M-1}{M-A-1}\binom{K-M}{R_{2}-M+A}$ times
to $c$. Let $m$ be the total number of vertices in $\M$, we have
\begin{eqnarray*}
T_{1}\left(\Delta_{A}\right) & = & T_{1}\left(\Gamma\right)\cdot\frac{m}{c}\binom{M-1}{M-A-1}\binom{K-M}{R_{2}-M+A}\\
 & = & \frac{rm}{s-A}\binom{M-1}{M-A-1}\binom{K-M}{R_{2}-M+A}
\end{eqnarray*}

Next, a vertex $v$ in row 1 of a column $\X$ contributes to $\ap+\cp+d$
if $\X\in\K\backslash\M$ and $\X$ is marked in row 2. As there are
$\binom{K-M-1}{R_{2}-M+A-1}$ ways to mark the columns of $\K\backslash\M$
in row 2 with $\X$ marked, and $\binom{M}{M-A}$ ways to mark the
columns of $\M$, $v$ contributes $\binom{M}{M-A}\binom{K-M-1}{R_{2}-M+A-1}$
times to $\ap+\cp+d$. Given that there are $s-m$ vertices in $\K\backslash\M$,
we have
\begin{eqnarray*}
T_{2}\left(\Delta_{A}\right) & = & \frac{r\left(s-m\right)}{s-A}\binom{M}{M-A}\binom{K-M-1}{R_{2}-M+A-1}
\end{eqnarray*}

Similarly, let $\left\{ v,u\right\} $ be a pair of vertices with
$v$ in row 1 of a column $\X$ and $u$ in row 2 of a column $\Y$.
Then, $\left\{ v,u\right\} $ contributes to $b\left(c+\cb+\cp\right)$
if the following conditions hold. First, we have $\X\in\R_{1}$, $\Y\in\K\backslash\R_{1}$,
and $\X$ unmarked in row 2. Furthermore, let $\Z$ be the column
$\Y$ if $\Y\in\M$, and $\Z$ be the column that $\Y$ points to
if $\Y\in\H$. Then, $\Z$ must be a column of $\M$ and must also
be marked. Now, as there are $\binom{M-1}{M-A-1}$ ways to mark the
columns of $\M$ with $\Z$ marked, and $\binom{K-M-1}{R_{2}-M+A}$
ways to mark the columns of $\K\backslash\M$ in row 2 with $\X$
unmarked, $\left\{ v,u\right\} $ contributes $\binom{M-1}{M-A-1}\binom{K-M-1}{R_{2}-M+A}$
times to $b\left(c+\cb+\cp\right)$. Given that there are $r\left(s-r\right)$
such pairs of $\left\{ v,u\right\} $, we have
\begin{eqnarray*}
T_{3}\left(\Delta_{A}\right) & = & \frac{r\left(s-r\right)}{\left(s-A\right)\left(s-A-1\right)}\binom{M-1}{M-A-1}\binom{K-M-1}{R_{2}-M+A}
\end{eqnarray*}

Finally, a vertex $v$ in row 1 of a column $\X$ contributes to $\cb$
if $\X\in\H$, $\X$ is unmarked in row 2, and the column $\Z$ that
$\X$ points to is marked in row 2. As there are $\binom{K-M-1}{R_{2}-M+A}$
ways to mark the columns of $\K\backslash\M$ in row 2 with $\X$
unmarked, and $\binom{M-1}{M-A-1}$ ways to mark the columns of $\M$
with $\Z$ marked, $v$ contributes $\binom{M-1}{M-A-1}\binom{K-M-1}{R_{2}-M+A-1}$
times to $\cb$. Given that there are $s-m-r$ vertices in $\H$,
we have
\begin{eqnarray*}
T_{4}\left(\Delta_{A}\right) & = & -\frac{r\left(s-m-r\right)}{\left(s-A\right)\left(s-A-1\right)}\binom{M-1}{M-A-1}\binom{K-M-1}{R_{2}-M+A}
\end{eqnarray*}

Now, let $T_{3+4}\left(\Delta_{A}\right)=T_{3}\left(\Delta_{A}\right)+T_{4}\left(\Delta_{A}\right)$,
and observe that 
\begin{eqnarray*}
T_{3+4}\left(\Delta_{A}\right) & = & \frac{rm}{\left(s-A\right)\left(s-A-1\right)}\binom{M-1}{M-A-1}\binom{K-M-1}{R_{2}-M+A}
\end{eqnarray*}
and 
\begin{eqnarray*}
T_{1}\left(\Delta_{A}\right)+T_{2}\left(\Delta_{A}\right) & = & T_{r}\left(\Delta_{A}\right)+T_{m1}\left(\Delta_{A}\right)+T_{m2}\left(\Delta_{A}\right)
\end{eqnarray*}
where
\begin{eqnarray*}
T_{r}\left(\Delta_{A}\right) & = & \frac{rs}{s-A}\binom{M}{M-A}\binom{K-M-1}{R_{2}-M+A-1}\\
T_{m1}\left(\Delta_{A}\right) & = & \frac{rm}{s-A}\binom{M-1}{M-A-1}\binom{K-M-1}{R_{2}-M+A}\\
T_{m2}\left(\Delta_{A}\right) & = & -\frac{rm}{s-A}\binom{M-1}{M-A}\binom{K-M-1}{R_{2}-M+A-1}
\end{eqnarray*}
By substituting these formulas into $T\left(\Delta\right)$, we have
\[
T\left(\Delta\right)=\left(s-1\right)!\left(\sum_{A=0}^{s-1}\left(T_{r}\left(\Delta_{A}\right)+T_{m1}\left(\Delta_{A}\right)+T_{m2}\left(\Delta_{A}\right)\right)+\sum_{A=0}^{s-2}T_{3+4}\left(\Delta_{A}\right)\right)
\]
Next, we will show that $\sum_{A=0}^{s-1}\left(T_{m1}\left(\Delta_{A}\right)+T_{m2}\left(\Delta_{A}\right)\right)+\sum_{A=0}^{s-2}T_{3+4}\left(\Delta_{A}\right)=0$.
Note that 
\begin{eqnarray*}
T_{m1}\left(\Delta_{A}\right)+T_{3+4}\left(\Delta_{A}\right) & = & \frac{rm}{s-A-1}\binom{M-1}{M-A-1}\binom{K-M-1}{R_{2}-M+A}
\end{eqnarray*}
for $0\le A\le s-2$. Therefore, by shifting the index of $\sum_{A=0}^{s-1}T_{m2}\left(\Delta_{A}\right)$
by one and noting that $T_{m2}\left(\Delta_{0}\right)=0$, we have
\begin{eqnarray*}
 &  & \sum_{A=0}^{s-1}\left(T_{m1}\left(\Delta_{A}\right)+T_{m2}\left(\Delta_{A}\right)\right)+\sum_{A=0}^{s-2}T_{3+4}\left(\Delta_{A}\right)\\
 & = & rm\binom{M-1}{M-s}\binom{K-M-1}{R_{2}-M+s-1}
\end{eqnarray*}
Now, for $\binom{M-1}{M-s}$ to be non-zero, we require $M\ge s$.
However, this implies that there are at least $s$ columns of $\M$,
each requiring a critical vertex. As there are only $s$ vertices
in row 1, $r$ is forced to be 0. Therefore, the entire sum is equal
to zero regardless of the value of $M$. Substituting this result
back into $T\left(\Delta\right)$, we obtain
\begin{eqnarray*}
T\left(\Delta\right) & = & \left(s-1\right)!\sum_{A=0}^{s-1}T_{r}\left(\Delta_{A}\right)\\
 & = & s!\sum_{A=0}^{s-1}\frac{r}{s-A}\binom{M}{M-A}\binom{K-M-1}{R_{2}-M+A-1}
\end{eqnarray*}
This proves our formula for $T\left(\Delta\right)$.
\end{proof}
Recall from \ref{def:Paired Array Conditions} that proper vertical
arrays do not require a vertex in each cell, so the formula in \ref{thm:Substructure Delta General Formula}
does not apply to all arrow arrays that will result in our subsequent
decomposition. In our previous paper \cite{Chan:2017-1}, which covers
the case $n=2$, we only needed the formula for two-row vertical arrays,
equivalently arrowed arrays without arrows. Hence, we bypassed this
issue by removing the columns with no vertices, then summed over all
possible ways to add the empty columns. However, that approach does
not work here, as arrowed arrays may have arrows in columns that are
otherwise empty. Therefore, we need to extend \ref{thm:Substructure Delta General Formula}
to cover a wider range of arrowed arrays.
\begin{defn}
\label{def:Admissible Substructure}An irreducible substructure $\Delta=\left(\mathbf{w},\R_{1},\phi\right)$
is \emph{admissible} if each cell that contains an arrow-head also
contains at least one vertex. This means that if an arrowed array
$\left(\alpha,\phi\right)$ satisfies an admissible substructure $\Delta$,
then $\psi_{i}\left(j\right)$ must have at least one vertex. In particular,
the only way to violate the forest condition of row $i$ is for there
to be a cycle in the functional digraph of $\psi_{i}$.
\end{defn}
Note that the definition of \emph{admissible} for substructure $\Delta$
is compatible with the definition of irreducible for substructure
$\Lambda$. That is, if $\Lambda=\left(\mathbf{x},\P,\phi\right)$
is an irreducible substructure and $\Delta=\left(\mathbf{w},\R_{1},\phi\right)$
is a refinement of $\Lambda$, then $\Delta$ can be reduced to an
admissible substructure $\Delta^{\prime}$ by the application of \ref{lemma:Arrow Simplification Substructure-1}.
Therefore, we will provide a formula for admissible substructure $\Delta$
as follows.
\begin{thm}
\label{thm:Substructure Delta Admissible Formula}Let $R_{1},R_{2}\ge1$,
and let $\Delta=\left(\mathbf{w},\R_{1},\phi\right)$ be an admissible
substructure. Then, the number of arrowed arrays $\left(\alpha,\phi\right)\in\AR{K}{R_{1}}{R_{2}}{s}$
with substructure $\Delta$ is given by the same formula as in \ref{thm:Substructure Delta General Formula}.
That is, 
\begin{eqnarray*}
T\left(\Delta\right) & = & s!\sum_{A=0}^{s-1}\frac{r}{s-A}\binom{M}{M-A}\binom{K-M-1}{R_{2}-M+A-1}
\end{eqnarray*}
\end{thm}
\begin{proof}
As permuting the columns of an arrowed array does not change whether
it satisfies the forest condition, we can without loss of generality
assume that the first $k$ of the $K$ columns of $\Delta$ are the
ones that contain at least one vertex. In particular, it means that
$\phi_{i}\left(j\right)\in\left[k\right]$. Now, let $\Delta^{R}$
be the subset of arrowed arrays that satisfies $\Delta$, and have
exactly $R$ marked cells in the first $k$ columns of row 2. Furthermore,
let $\Delta^{R;k}=\left(\mathbf{w}^{\prime},\R_{1}\cap\left[k\right],\phi^{\prime}\right)$
be the restriction of $\Delta^{R}$ to the first $k$ columns. In
other words, $\Delta^{R;k}=\left(\mathbf{w}^{\prime},\R_{1}\cap\left[k\right],\phi^{\prime}\right)$
is a substructure of $\AR{k}{\left|\R_{1}\cap\left[k\right]\right|}{R}{s}$,
where $w_{j}^{\prime}=w_{j}$ and $\phi_{i}^{\prime}\left(j\right)=\phi_{i}\left(j\right)$
for $1\le j\le k$. Note that $\phi_{i}\left(j\right)\in\left[k\right]$
implies that $\phi_{i}^{\prime}\left(j\right)\in\left[k\right]$,
so this is well defined. We will show that there is a $\binom{K-k}{R_{2}-R}$
to 1 correspondence between arrowed arrays satisfying $\Delta^{R}$
and arrowed arrays satisfying $\Delta^{R;k}$.

Let $\left(\alpha,\phi\right)$ be an arrowed array satisfying $\Delta^{R}$
and consider the cell $\left(i,j\right)$, where $k+1\le j\le K$.
As $\Delta^{R}$ is admissible, there cannot be another column $j^{\prime}$
such that $\psi_{i}\left(j^{\prime}\right)=j$. So, by deleting this
column, we have either deleted an isolated root vertex, deleted a
leaf, or done nothing to the functional digraph of $\psi_{i}$. Hence,
we can remove the column $j$ from the array without violating the
forest condition. Therefore, we can simply cut off the rightmost $K-k$
columns of $\left(\alpha,\phi\right)$ to obtain an arrowed array
$\left(\alpha^{\prime},\phi^{\prime}\right)$ that satisfies $\Delta^{R;k}$.

Conversely, given an arrowed array $\left(\alpha^{\prime},\phi^{\prime}\right)$
satisfying $\Delta^{R;k}$, we can add $K-k$ columns with no vertices
to obtain an arrowed array $\left(\alpha,\phi\right)$ satisfying
$\Delta^{R}$. Note that the positions of arrows and marked cells
in row 1 is completely fixed by $\Delta^{R}$. However, only the first
$k$ columns of $\left(\alpha,\phi\right)$ are predetermined in row
2, as given by $\left(\alpha^{\prime},\phi^{\prime}\right)$. For
the remaining $K-k$ columns, we can mark $R_{2}-R$ cells arbitrarily
and satisfy the forest condition, as adding columns with no vertices
does not change $\psi_{2}$. Therefore, for each arrowed array $\left(\alpha^{\prime},\phi^{\prime}\right)$
satisfying $\Delta^{R;k}$, there are exactly $\binom{K-k}{R_{2}-R}$
arrowed arrays satisfying $\Delta^{R}$.

By construction, each of the $\Delta^{R;k}$ has $w_{j}>0$ for $1\le j\le k$,
so we can use \ref{thm:Substructure Delta General Formula} to obtain
$T\left(\Delta^{R;k}\right)$. Furthermore, $\Delta^{1},\dots,\Delta^{\min\left(k,R_{2}\right)}$
partitions $\Delta$, and for $R=0$ or $R>k$, we have $T\left(\Delta^{R;k}\right)=0$.
Therefore, we can change the bounds to $0\le k\le R_{2}$, and use
the Chu-Vandermonde identity (pg. 67 of \cite{Andrews-Askey-Roy:1999})
to obtain 
\begin{eqnarray*}
T\left(\Delta\right) & = & \sum_{R=0}^{\min\left(k,R_{2}\right)}T\left(\Delta^{R;k}\right)\binom{K-k}{R_{2}-R}\\
 & = & s!\sum_{A=0}^{s-1}\sum_{R=0}^{R_{2}}\frac{r_{2}}{s-A}\binom{M}{M-A}\binom{k-M-1}{R-M+A-1}\binom{K-k}{R_{2}-R}\\
 & = & s!\sum_{A=0}^{s-1}\frac{r_{2}}{s-A}\binom{M}{M-A}\binom{K-M-1}{R_{2}-M+A-1}
\end{eqnarray*}
which is the formula for $T\left(\Delta\right)$ as given by \ref{thm:Substructure Delta General Formula}.
\end{proof}
Next, we will rewrite this formula using hypergeometric transformations,
as that will simplify our work later.
\begin{thm}
\label{thm:Substructure Delta Admissible Formula 2}Let $R_{1},R_{2}\ge1$,
and let $\Delta=\left(\mathbf{w},\R_{1},\phi\right)$ be an admissible
substructure that satisfies the balance condition. Then, the number
of arrowed arrays $\left(\alpha,\phi\right)\in\AR{K}{R_{1}}{R_{2}}{s}$
with substructure $\Delta$ is given by the formula 
\[
T\left(\Delta\right)=r\sum_{A=0}^{\min\left(s,K\right)-1}\frac{M!\left(K-A-1\right)!\left(s-A-1\right)!}{\left(M-A\right)!\left(K-R_{2}-A\right)!\left(R_{2}-1\right)!}
\]
where $r$ is the total number of vertices in row 1 of the columns
of $\R_{1}$, and $M$ is the number of columns that contain a critical
vertex in row 1.
\end{thm}
\begin{proof}
First, we rewrite $T\left(\Delta\right)$ using factorials to obtain
\[
T\left(\Delta\right)=r\sum_{A=0}^{s-1}\frac{s!M!\left(s-A-1\right)!\left(K-M-1\right)!}{\left(s-A\right)!\left(M-A\right)!A!\left(R_{2}-M+A-1\right)!\left(K-R_{2}-A\right)!}
\]
If $M\ge s$, then $r=0$, as each column of $\M$ requires a critical
vertex, and there are only $s$ vertices in row 1. In this case, the
theorem is true as both the original formula and the new formula imply
that $T\left(\Delta\right)=0$. Otherwise, we have $M\le s-1$ and
$\left(M-A\right)!$ in the denominator, which allows us to lower
the upper bound of the summation to $M$. We can then write it using
the standard notation for hypergeometric series to obtain
\begin{eqnarray*}
T\left(\Delta\right) & = & r\cdot{}_{3}F_{2}\left({-M,-s,-K+R_{2}\atop R_{2}-M,-s+1};1\right)\frac{\left(s-1\right)!\left(K-M-1\right)!}{\left(R_{2}-M-1\right)!\left(K-R_{2}\right)!}\\
 & = & r\cdot{}_{3}F_{2}\left({-M,1,-K+R_{2}\atop 1-K,-s+1};1\right)\frac{\left(K-M\right)^{\left(M\right)}\left(s-1\right)!\left(K-M-1\right)!}{\left(R_{2}-M\right)^{\left(M\right)}\left(R_{2}-M-1\right)!\left(K-R_{2}\right)!}\\
 & = & r\sum_{A=0}^{\min\left(s,K\right)-1}\frac{M!\left(K-A-1\right)!\left(s-A-1\right)!}{\left(M-A\right)!\left(K-R_{2}-A\right)!\left(R_{2}-1\right)!}
\end{eqnarray*}
where we use the $_{3}F_{2}$ identity 
\[
_{3}F_{2}\left({-N,b,c\atop d,e};1\right)=\frac{\left(d-c\right)^{\left(N\right)}}{d^{\left(N\right)}}{}_{3}F_{2}\left({-N,e-b,c\atop 1-N-d+c,e};1\right)
\]
for non-negative integer $N$, and $a,b,c,d\in\mathbb{C}$. This identity
can be found on pg. 142 of \cite{Andrews-Askey-Roy:1999}.

Now, as $\left(M-A\right)!$ is again part of the new denominator,
we can raise the summation index without changing the value of the
sum. Note that we know $M\le s-1$, and we can deduce that $M\le K-1$
as $R_{1}\ge1$. This allows us to raise the upper bound to $\min\left(s,K\right)-1$,
while keeping the numerator well defined.
\end{proof}
The benefit of this new formula is that we are no longer required
to keep $M\le\min\left(s,K\right)-1$. While taking $M\ge\min\left(s,K\right)$
for $\Delta$ makes no sense combinatorially, the value for $T\left(\Delta\right)$
is well defined and finite. This frees up $M$ for manipulation and
summation if we can multiply $T\left(\Delta\right)$ with an expression
that is zero if $M\ge s$ or $M\ge K$. When we do the induction on
the number of vertical arrays, this fact will become extremely useful.

With the formula for admissible substructures $\Delta$, we can now
provide a formula for the number of arrowed arrays satisfying substructure
$\Lambda$.
\begin{thm}
\label{thm:Substructure Lambda Formula}Given a substructure $\Lambda=\left(\mathbf{x},\P,\phi\right)$
such that the functional digraph of $\phi$ on $\H\cup\P$ is a rooted
forest with root vertices $\P$, the number of arrowed arrays $\left(\alpha,\phi\right)\in\AR{K}{R_{1}}{R_{2}}{s}$
satisfying substructure $\Lambda$ is given by the formula
\[
T\left(\Lambda\right)=\sum_{A=0}^{\min\left(s,K\right)-1}\frac{\left(s-P+R_{1}\right)\left(K-A-1\right)!\left(s-A-1\right)!\left(P-1\right)!}{\left(P-R_{1}-A\right)!\left(K-R_{2}-A\right)!\left(R_{1}-1\right)!\left(R_{2}-1\right)!}
\]
where $P$ is the number of columns of $\P$.
\end{thm}
\begin{proof}
First, we suppose that $\Lambda$ is irreducible. We prove this by
substituting into the formula for $T\left(\Delta\right)$ given by
\ref{thm:Substructure Delta Admissible Formula 2}. Let $\R_{1}$
be an $R_{1}$-subset of $\P$, and consider the substructure $\Delta=\left(\mathbf{x},\R_{1},\phi\right)$,
where $x_{i}^{\prime}=x_{i}+1$ if $x\in\P\backslash\R_{1}$, and
$x_{i}^{\prime}=x_{i}$, otherwise. Now, note that $\Delta$ may not
be irreducible, as there can be arrows pointing to the columns of
$\R_{1}$. Therefore, we have to reduce $\Delta$ using the arrow
simplification lemma defined in \ref{lemma:Arrow Simplification Substructure-1}.
This gives us an irreducible substructure $\Delta^{\prime}=\left(\mathbf{x}^{\prime},\R_{1}\cup\H_{1},\phi^{\prime}\right)$,
where $\H_{1}\subseteq\H$ is the set of columns that points to $\R_{1}$,
and $\phi^{\prime}$ is $\phi$ restricted to the columns of $\H\backslash\H_{1}$.

Now, $\Delta^{\prime}$ satisfies the balance condition by construction.
Furthermore, any cell of $\Delta^{\prime}$ that contains an arrow-head
must be in $\P\backslash\R_{1}$, as otherwise $\Delta^{\prime}$
will not be irreducible. Since the columns of $\P\backslash\R_{1}$
must each contain at least one vertex, $\Delta^{\prime}$ is an admissible
substructure, so we can use the formula for $T\left(\Delta\right)$
given by \ref{thm:Substructure Delta Admissible Formula 2}. As $\Delta^{\prime}$
satisfies the balance condition, we can take $r$ to be the number
of vertices in row 1 of $\R_{1}$. Observe that the $P-R_{1}$ vertices
added to row 1 of $\P\backslash\R_{1}$ are all critical vertices,
regardless of the choice of $\R_{1}$. Hence, they never contribute
to $T\left(\Delta\right)$. This means that we only need to consider
the non-critical vertices of row 1, which are given by $\mathbf{x}$.
Now, a non-critical vertex $u$ in row 1 of a column $\X$ contributes
to $r$ of the formula for $T\left(\Delta\right)$ if $\X\in\R_{1}$,
or $\X\in\H$ and $\X$ points to a column in $\R_{1}$. In either
case, there are $\binom{P-1}{R_{1}-1}$ different subsets $\R_{1}$
such that $\X$ is marked in $\Delta^{\prime}$, out of the $\binom{P}{R_{1}}$
possible $R_{1}$-subsets of $\P$. Given that all non-critical vertices
of row 1 are in $\P\cup\H$, and that there are $s-P+R_{1}$ non-critical
vertices in row 1, we have
\begin{eqnarray*}
T\left(\Lambda\right) & = & T\left(\Delta^{\prime}\right)\cdot\frac{s-P+R_{1}}{r}\binom{P-1}{R_{1}-1}\\
 & = & \sum_{A=0}^{\min\left(s,K\right)-1}\frac{\left(s-P+R_{1}\right)\left(K-A-1\right)!\left(s-A-1\right)!\left(P-1\right)!}{\left(P-R_{1}-A\right)!\left(K-R_{2}-A\right)!\left(R_{1}-1\right)!\left(R_{2}-1\right)!}
\end{eqnarray*}
where we substitute in $M=P-R_{1}$ as the number of critical vertices
in row 1.

Finally, if $\Lambda$ is not irreducible, we can repeatedly apply
\ref{lemma:Arrow Simplification Substructure-2} to obtain an irreducible
substructure $\Lambda^{\prime}=\left(\mathbf{y},\P,\phi^{\prime}\right)$.
As $s$, $K$, $R_{1}$, $R_{2}$, and $\P$ all remain the same,
we have $T\left(\Lambda\right)=T\left(\Lambda^{\prime}\right)$, so
the result follows.
\end{proof}

\section{\label{sec:Vertical Array Decomposition}Enumeration of Vertical
Arrays}

At this point, we are ready to decompose proper vertical arrays. Recall
that a paired array $\alpha\in\PA{n}{K}{\mathbf{R}}{\mathbf{q}}{\mathbf{s}}$
is tree-shaped if the support graph of $\mathbf{s}$ is a tree. With
tree-shaped vertical arrays, we can delete a row that is a leaf in
the support graph while keeping the support graph a tree. This allows
us to recursively decompose tree-shaped vertical arrays into smaller
tree-shaped vertical arrays and arrowed arrays. Then, by using \ref{thm:Substructure Lambda Formula},
we can provide a formula for $v_{n,K;\mathbf{R}}^{\left(\mathbf{s}\right)}$
when the support graph of $\mathbf{s}$ is a tree.

We start off with a number of preliminary definitions and facts.
\begin{fact}
\label{fact:Proper Array Marking Cells}Let $\alpha\in\PA{n}{K}{\mathbf{R}}{\mathbf{q}}{\mathbf{s}}$
be a proper paired array. Suppose cell $\left(i,j\right)$ of $\alpha$
is an unmarked cell containing at least one vertex, then $\alpha^{\prime}$
that is formed by marking cell $\left(i,j\right)$ of $\alpha$ is
also a proper paired array.
\end{fact}
Note that the converse of \ref{fact:Proper Array Marking Cells} is
not true. For example, if $\alpha^{\prime}$ has only one marked cell
in row $i$, then unmarking that cell violates the forest condition
for that row. This fact allows us to mark cells containing critical
vertices, making those vertices non-critical and removing them from
the forest condition. This leads to our next definition.
\begin{defn}
\label{def:Partially Paired Array}If $n,K\ge1$, then a \emph{partially-paired
array} $\alpha$ is an $n\times K$ array of cells, where each cell
contains zero or more vertices, and is either marked or unmarked.
Furthermore, each vertex of the array may be paired with another vertex.
However, only the rightmost vertices of unmarked cells are required
to be paired with another vertex, and we call the vertices not paired
with any other vertices \emph{unpaired vertices}. Terms for paired
arrays such as \emph{critical vertices} and parameters like $q_{i}$
and $R_{i}$ carry over from \ref{def:Paired Array} and \ref{conv:Array Convention}.
\end{defn}
By definition, all paired arrays are partially-paired arrays. Also,
as unpaired vertices are neither mixed nor critical, they do not affect
the balance or forest conditions. However, we do consider unpaired
vertices as objects in a partially-paired array.

Now, our main reason for using partially-paired arrays is so that
we can unpair vertices of a paired array. That is, if $\left\{ u,v\right\} $
is a pair of non-critical vertices in a partially-paired array $\alpha$,
we can unpair them to create a new partially-paired array $\alpha^{\prime}$
that is otherwise identical to $\alpha$, but with $u$ and $v$ unpaired.
Then, we can remove $u$ and $v$ separately without impacting the
balance and forest conditions. We will adapt a technique from Goulden
and Slofstra for labelling the objects in a row of a partially-paired
array with a set of positive integers. This allows us to insert or
remove a subset of the unpaired vertices while keeping track of their
positions.
\begin{proc}
\label{proc:Labelling Procedure}Let $\alpha$ be a partially-paired
array with $p_{i}$ vertices and $R_{i}$ marked cells in row $i$,
where $1\le i\le n$. We describe the following three procedures:

\begin{enumerate}
\item Let $\mathcal{S}$ be a set of positive integers of size $p_{i}+R_{i}$.
To \emph{label row $i$ of $\alpha$ with $\mathcal{S}$} is to assign
from left to right elements of $\mathcal{S}$ to the objects of row
$i$, from smallest to largest. As described in \ref{def:Paired Array},
in a cell that contains both vertices and a box, the box is to be
taken as the rightmost object of the cell.
\item Let $\V$ be a subset of the unpaired vertices in row $i$. To \emph{extract
$\V$ from $\alpha$} is to create a partially-paired array $\alpha^{\prime}$
and a set of positive integers $\W$, where $\alpha^{\prime}$ is
$\alpha$ with $\V$ deleted, and $\W$ is a $\left|\V\right|$-subset
of $\left[p_{i}+R_{i}-1\right]$. This is done by labelling row $i$
of $\alpha$ with $\left[p_{i}+R_{i}\right]$, then deleting $\V$
from $\alpha$. We let $\W$ be the labels of the vertices deleted.
Note that $\W$ cannot contain $p_{i}+R_{i}$ as the deleted vertices
cannot be the rightmost objects of their cells.
\item Let $\W$ be a $y$-subset of $\left[p_{i}+R_{i}+y-1\right]$, where
$y\ge0$. To \emph{insert $\W$ into row $i$ of $\alpha$} is to
add $y$ unpaired vertices to row $i$ of $\alpha$ to create a partially-paired
array $\alpha^{\prime}$. This is done by labelling row $i$ of $\alpha$
with $\left[p_{i}+R_{i}+y\right]\backslash\W$. Then, for each $w\in\W$,
we find the smallest $w^{\prime}\notin\W$ such that $w^{\prime}>w$,
and place a vertex to the left of and in the same cell as the object
labelled $w^{\prime}$. As the new vertex is not the rightmost object
of a cell, it is non-critical. Furthermore, if there is more than
one vertex to be inserted to the left of an object, they should be
inserted in increasing order from left to right. In the end, row $i$
of $\alpha^{\prime}$ contains $p_{i}+R_{i}+y$ objects, labelled
from left to right by $1$ to $p_{i}+R_{i}+y$ in increasing order.
Finally, we let $\V$ denote the set of vertices inserted, to mirror
the extraction procedure.
\end{enumerate}
\end{proc}
Notice that in both the extraction and insertion procedures, the vertices
involved are unpaired. Furthermore, the use of the same variables
$\V$ and $\W$ between procedure 2 and 3 is deliberate, as we shall
now show that the extraction and insertion procedures are inverses
of each other.
\begin{prop}
\label{prop:Extraction/insertion inverses}Let $\alpha$ be a partially-paired
array with $p_{i}$ vertices and $R_{i}$ marked cells in row $i$,
and $\V$ be a subset of the unpaired vertices in row $i$, where
$1\le i\le n$. Let $\beta$ be the partially-paired array and $\W$
be the subset of $\left[p_{i}+R_{i}-1\right]$ created by extracting
$\V$ from $\alpha$. Suppose $\alpha^{\prime}$ is the partially-paired
array formed by reinserting $\W$ into row $i$ of $\beta$, and $\V^{\prime}$
is the set of vertices inserted, then $\alpha=\alpha^{\prime}$ and
$\V=\V^{\prime}$. Conversely, let $\beta$ be a partially-paired
array with $p_{i}$ vertices and $R_{i}$ marked cells in row $i$,
where $1\le i\le n$, and suppose $\W$ is a $y$-subset of $\left[p_{i}+R_{i}+y-1\right]$,
with $y\ge0$. Let $\alpha$ be the partially-paired array formed
by inserting $\W$ into row $i$ of $\beta$, and $\V$ be the set
of inserted vertices. Suppose $\beta^{\prime}$ and $\W^{\prime}$
is the pair of objects created from extracting $\V$ from $\alpha$,
then $\beta=\beta^{\prime}$ and $\W=\W^{\prime}$. In both cases,
$\alpha$ is proper if and only if $\beta$ is proper.
\end{prop}
\begin{proof}
Note that when we extract $\V$ from $\alpha$, we obtain the partially-paired
array $\beta$ and the set $\W$ that is a $\left|\V\right|$-subset
of $\left[p_{i}+R_{i}-1\right]$. As $\beta$ has $p_{i}+R_{i}-\left|\V\right|$
objects in row $i$, we can insert $\W$ into $\beta$ to obtain the
partially-paired array $\alpha^{\prime}$ and the set $\V^{\prime}$
of inserted vertices. Furthermore, the objects remaining in $\beta$
are labelled with the same labels $\left[p_{i}+R_{i}-1\right]\backslash\W$
during the extraction and insertion procedures. Finally, each vertex
$v\in\V$ is in the same cell and to the left of some other object
in $\alpha$, and is reinserted into that same cell in $\alpha^{\prime}$
in increasing order of labels. Therefore, $\alpha=\alpha^{\prime}$
and $\V=\V^{\prime}$.

Conversely, when we insert $\W$ into row $i$ of $\beta$, the vertices
inserted by $\W$ are non-critical vertices, and are the objects in
$\alpha$ that are labelled from left to right with $\left[p_{i}+R_{i}+\left|\W\right|\right]$.
As the set of inserted vertices retains the same labels when in the
extraction procedure, we have $\beta=\beta^{\prime}$ and $\W=\W^{\prime}$.

Finally, in both the extraction and insertion procedures, the vertices
involved are non-critical and unpaired. Therefore, they do not impact
the balance or the forest conditions. Hence, $\alpha$ is proper if
and only if $\beta$ is proper.
\end{proof}
Next, we define the compatibility condition that allows us to combine
arrowed arrays and vertical arrays together.
\begin{defn}
\label{def:Lambda Compatible}Let $\alpha\in\PVA{n}{K}{\mathbf{R}}{\mathbf{s}}$
be an $n$-row proper vertical array with $\R_{i}$ as its set of
marked cells in row $i$, and $\psi_{i}$ as its forest condition
function for row $i$. A substructure $\Lambda=\left(\mathbf{x},\P,\phi\right)$
as defined in \ref{def:Substructure} is \emph{$\Lambda$-compatible}
with row $i$ of $\alpha$ if $\P=\R_{i}$ and $\phi=\psi_{i}$. Furthermore,
let $R_{1}^{\prime}$ and $R_{2}^{\prime}$ be such that $1\le R_{1}^{\prime}\le R_{i}$
and $1\le R_{2}^{\prime}\le K$, and suppose that $\W$ is a $x$-subset
of $\left[s_{i}+R_{i}+x-1\right]$ for some $x\ge0$. We define $\Lambda_{\alpha,i,\W}$
to be the substructure of $\AR{K}{R_{1}^{\prime}}{R_{2}^{\prime}}{x+R_{i}-R_{1}^{\prime}}$
with parameters $\Lambda_{\alpha,i,\W}=\left(\mathbf{x},\R_{i},\psi_{i}\right)$,
where $\mathbf{x}=\left(x_{1},\dots,x_{K}\right)$ and $x_{j}$ is
the number of vertices inserted into cell $\left(i,j\right)$ of $\alpha$
if $\W$ is inserted into row $i$ of $\alpha$ by the insertion procedure
defined in \ref{proc:Labelling Procedure}.
\end{defn}
By definition, $\Lambda_{\alpha,i,\W}$ is $\Lambda$-compatible with
row $i$ of $\alpha$. Also, by summing over the number of vertices
inserted into cell $\left(i,j\right)$, we have $\left|\W\right|=\sum_{j}x_{j}$.

With substructure compatibility defined, we can now decompose tree-shaped
vertical arrays. Let $\alpha\in\PVA{n+1}{K}{\mathbf{R}}{\mathbf{s}}$
be an $\left(n+1\right)$-row proper vertical array, and without loss
of generality assume that row $n+1$ is a leaf vertex adjacent to
row $n$ in the support graph of $\mathbf{s}$. To extract row $n+1$
from $\alpha$, we mark the cells in row $n$ containing the critical
vertices matched with vertices in row $n+1$. Then, we remove all
pairs between rows $n$ and $n+1$, and subsequently delete row $n+1$.
To keep track of the removed vertices in row $n$, we use a $\left(s_{n+1}-P+R_{n}\right)$-subset
to represent the positions of the non-critical vertices, and an arrowed
array to represent the critical vertices and pairings of the vertices
removed.
\begin{thm}
\label{thm:Vertical Array Decomposition}Let $n,K\ge1$, $\mathbf{s}=\left(s_{1,2},s_{1,3},\dots,s_{n,n+1}\right)\ge\mathbf{0}$,
and $\mathbf{R}=\left(R_{1},\dots,R_{n+1}\right)\in\left[K\right]^{n+1}$.
Suppose the support graph of $\mathbf{s}$ is a tree with the vertex
$n+1$ as a leaf adjacent to the vertex $n$. Then, there exists a
decomposition
\[
\zeta\colon\PVA{n+1}{K}{\mathbf{R}}{\mathbf{s}}\rightarrow\bigcup_{P=R_{n}}^{\min\left(s_{n+1}+R_{n},K\right)}\bigcup_{\substack{\beta\in\PVA{n}{K}{\mathbf{R}^{\prime}}{\mathbf{s}^{\prime}}\\
\W\in\ksub{s_{n}+R_{n}-1}{s_{n+1}-P+R_{n}}
}
}\left(\beta,\W,\Lambda_{\beta,n,\W}\right)
\]
of proper vertical arrays into a triple of smaller vertical arrays,
$\left(s_{n+1}-P+R_{n}\right)$-subsets, and arrowed arrays. Here,
$\Lambda_{\beta,n,\W}$ are substructures of $\AR{K}{R_{n}}{R_{n+1}}{s_{n+1}}$,
$\mathbf{s}^{\prime}$ is $\mathbf{s}$ restricted to an $n\times n$
matrix by removing the last row and column, $s_{i}=\sum_{k\neq i}s_{i,k}$
for $1\le i\le n+1$, and $\mathbf{R}^{\prime}$ is a vector of length
$n$ given by 
\begin{eqnarray*}
R_{k}^{\prime} & = & \begin{cases}
R_{k} & k<n\\
P & k=n
\end{cases}
\end{eqnarray*}
Furthermore, this decomposition is a bijection.
\end{thm}
Note that we can apply this theorem to any pair of rows $i$ and $k$,
such that $i$ is a leaf in the support graph. Also, $s_{n}$ includes
the vertex pairs between rows $n$ and $n+1$, and the marked cells
in row $n$ of $\beta$ are given by $\R_{n}^{\prime}$, which is
a set of size $P$ that contains $\R_{n}$ as a subset.
\begin{proof}
We will prove the bijection by providing the decomposition and show
that it is invertible. Conceptually, we take the mixed pairs between
row $n$ and row $n+1$ of $\alpha$, and put them into an arrowed
array $\left(\sigma,\phi\right)$. Then, we add marked cells and arrows
to $\left(\sigma,\phi\right)$ in such a way that rows $n$ and $n+1$
of $\alpha$ have the same forest condition functions as rows 1 and
2 of $\left(\sigma,\phi\right)$, respectively. To record the position
of the non-critical vertices in row $n$, we extract and record these
vertices as a $s_{n+1}-P+R_{n}$-subset of $\left[s_{n}+R_{n}-1\right]$.
Finally, we mark the cells of $\alpha$ containing the critical vertices
of row $n$ that are paired with vertices of row $n+1$, so as to
preserve the forest condition for row $n$.

Let $\V$ be the set of non-critical vertices that are paired with
vertices of row $n+1$, and $\U$ be the set of critical vertices
that are paired with vertices of row $n+1$. Note that the vertices
of $\U$ and $\V$ must be in row $n$ by our assumption, and that
$\left|\U\cup\V\right|=s_{n+1}$. Therefore, if we let $P=R_{n}+\left|\U\right|$,
we have $R_{n}\le P\le K$. Furthermore, since $\left|\U\right|\le s_{n+1}$,
we have $P\le s_{n+1}+R_{n}$, which combines to give $R_{n}\le P\le\min\left(s_{n+1}+R_{n},K\right)$.

To construct the proper vertical array $\beta\in\PVA{n}{K}{\mathbf{R}^{\prime}}{\mathbf{s}^{\prime}}$
and the subset $\W\in\ksub{s_{n}+R_{n}-1}{s_{n+1}-P+R_{n}}$, we first
mark the cells containing the vertices of $\U$. Next, we unpair all
vertex pairs with one vertex in row $n+1$, delete row $n+1$, and
call the resulting array $\alpha^{\prime}$. As this leaves all other
mixed pairs unchanged, $\mathbf{s}^{\prime}$ describes the number
of mixed pairs of $\alpha^{\prime}$. Then, as the support graph of
$\mathbf{s}^{\prime}$ is the support graph of $\mathbf{s}$ with
the vertex $n+1$ removed, the support graph of $\mathbf{s}^{\prime}$
is also a tree. Also, note that deleting row $n+1$ removes the variables
$s_{n+1,k,j}$ and $s_{k,n+1,j}$ from $\alpha$, but leaves the remaining
$s_{i,k,j}$ the same for all $1\le i,k\le n$, $i\neq k$. Therefore,
the conditions of \ref{lem:tree shape balance} remain satisfied in
$\alpha^{\prime}$, so $\alpha^{\prime}$ satisfies the balance condition.
In addition, since we marked the cells containing $\U$, the forest
condition remains satisfied when we unpair the vertices of $\U\cup\V$
and delete row $n+1$. This means that $\alpha^{\prime}$ is a proper
partially-paired array. Next, we remove the vertices of $\U$ from
$\alpha^{\prime}$ to obtain the partially-paired array $\alpha^{\prime\prime}$,
and we extract $\V$ from $\alpha^{\prime\prime}$ as described in
\ref{proc:Labelling Procedure} to obtain the subset $\W$ and the
vertical array $\beta$. Note that $\alpha^{\prime\prime}$ has $R_{n}+\left|\U\right|$
marked cells, $s_{n}-\left|\U\right|$ total vertices, and $\left|\V\right|=s_{n}-P+R_{n}$
unpaired vertices in row $n$. Therefore, $\W$ is a $s_{n}-P+R_{n}$-subset
of $s_{n}+R_{n}-1$. Furthermore, by \ref{prop:Extraction/insertion inverses},
$\beta$ is also a proper paired array. By construction, $\beta$
satisfies the parameters $\mathbf{R}^{\prime}$ and $\mathbf{s}^{\prime}$,
and contains no non-mixed pairs, so $\beta\in\PVA{n}{K}{\mathbf{R}^{\prime}}{\mathbf{s}^{\prime}}$
as desired.

To preserve information on the pairs we removed, we construct an arrowed
array $\left(\sigma,\phi\right)\in\Lambda_{\beta,n,\W}$ such that
$\psi_{n}=\psi_{1}^{\prime}$ and $\psi_{n+1}=\psi_{2}^{\prime}$,
where $\psi_{n}$ and $\psi_{n+1}$ are the forest condition functions
for rows $n$ and $n+1$ of $\alpha$, while $\psi_{1}^{\prime}$
and $\psi_{2}^{\prime}$ are the forest condition functions for rows
1 and 2 of $\left(\sigma,\phi\right)$, respectively. For each vertex
$v\in\U\cup\V$ that is in cell $\left(n,j\right)$, we place a corresponding
$x_{v}$ into cell $\left(1,j\right)$ of $\sigma$. Similarly, for
each vertex $u$ in cell $\left(n+1,j\right)$, we place a corresponding
vertex $x_{u}$ in cell $\left(2,j\right)$ of $\sigma$. If we need
to place more than one vertex into the same cell, we place them in
the same order in $\sigma$ as they are in $\alpha$. Then, for each
pair $\left\{ v,u\right\} $ between row $n$ and $n+1$, we pair
their corresponding vertices $x_{u}$ and $x_{v}$ in $\sigma$. Next,
we mark cell $\left(1,j\right)$ of $\sigma$ if cell $\left(n,j\right)$
of $\alpha$ is marked, and we mark cell $\left(2,j\right)$ of $\sigma$
if cell $\left(n+1,j\right)$ of $\alpha$ is marked. Finally, suppose
$\left(n,j\right)$ of $\alpha$ contains a critical vertex $u\notin\U$.
Then, it must be paired with some vertex $v$ in some cell $\left(k,j^{\prime}\right)$,
where $1\le k\le n-1$. In this case, we let $\phi\left(j\right)=j^{\prime}$.
This completes the construction of $\left(\sigma,\phi\right)$.

By construction, $\left(\sigma,\phi\right)$ is in $\AR{K}{R_{n}}{R_{n+1}}{s_{n+1}}$,
as we copied all marked cells and vertex pairs between rows $n$ and
$n+1$ to $\left(\sigma,\phi\right)$. This also implies that cell
$\left(1,j\right)$ of $\sigma$ has $s_{n,n+1,j}$ vertices, and
cell $\left(2,j\right)$ of $\sigma$ has $s_{n+1,n,j}$ vertices.
By \ref{lem:tree shape balance}, we have $s_{n,n+1,j}=s_{n+1,n,j}$
for all $j$, so $\left(\sigma,\phi\right)$ satisfies the balance
condition. Furthermore, by replacing the critical pairs of row $n$
of $\alpha$ with arrows, we ensure that rows 1 and 2 of $\left(\sigma,\phi\right)$
have the same set of marked cells and forest condition functions as
rows $n$ and $n+1$ of $\alpha$. Therefore, $\left(\sigma,\phi\right)$
satisfies the forest condition, as $\alpha$ is a proper vertical
array.

Finally, we need to show that $\left(\sigma,\phi\right)\in\Lambda_{\beta,n,\W}=\left(\mathbf{x},\R_{n}^{\prime},\theta_{n}\right)$,
where $\theta_{n}$ is the forest condition function for row $n$
of $\beta$. By construction, $\theta_{n}\left(j\right)$ and $\phi\left(j\right)$
are only defined for cells $\left(n,j\right)$ with critical vertices
$u\notin\U$. Furthermore, we have $\theta_{n}\left(j\right)=\phi\left(j\right)$
in those cases, so $\theta_{n}=\phi$. Then, note that the set of
marked cells in row 1 of $\sigma$ is $\R_{n}$, which is a subset
of $\R_{n}^{\prime}$. Now, if we reinsert $\W$ into row $n$ of
$\beta$, we recover $\alpha^{\prime\prime}$ and the set $\V$ of
extracted vertices by \ref{prop:Extraction/insertion inverses}. These
are all the non-critical vertices in row 1 of $\left(\sigma,\phi\right)$,
in the same cells as $\V$. Furthermore, if cell $\left(n,j\right)$
is marked in $\beta$, then it must either be marked in $\alpha$,
in which case cell $\left(1,j\right)$ is marked in $\sigma$, or
contain a vertex $u\in\U$, in which case it is unmarked and contains
the vertex $x_{u}$. In both cases, $\left(\sigma,\phi\right)$ satisfies
$\mathbf{x}$ and $\R_{n}^{\prime}$. Therefore, we have $\left(\sigma,\phi\right)\in\Lambda_{\beta,n,\W}$,
as desired.

Conversely, let $\beta\in\PVA{n}{K}{\mathbf{R}^{\prime}}{\mathbf{s}^{\prime}}$,
$\W\in\ksub{s_{n}+R_{n}-1}{s_{n+1}-P+R_{n}}$, and $\left(\sigma,\phi\right)\in\AR{K}{R_{n}}{R_{n+1}}{s_{n+1}}$
that satisfies $\Lambda_{\beta,n,\W}$. We first construct partially-paired
array $\beta^{\prime}$ by inserting $\W$ into row $n$ of $\beta$
as described in \ref{proc:Labelling Procedure}. This gives us a set
$\V$ of unpaired vertices in $\beta^{\prime}$, labelled with the
elements of $\W$. By \ref{prop:Extraction/insertion inverses}, $\beta^{\prime}$
is a proper partially-paired array. Furthermore, by the definition
of $\Lambda_{\beta,n,\W}$, the vertices of $\V$ are in the same
columns as the non-critical vertices in row 1 of $\sigma$. Therefore,
for each vertex $v\in\V$, we can let $x_{v}$ be the non-critical
vertex in row 1 of $\sigma$ that corresponds to $v$. Next, consider
each cell $\left(1,j\right)$ of $\sigma$ that contains a critical
vertex. Since $\Lambda_{\beta,n,\W}$ is $\Lambda$-compatible with
$\beta$, cell $\left(n,j\right)$ is marked in $\beta$, and by extension
$\beta^{\prime}$. This means that we can add an unpaired vertex $u$
to cell $\left(n,j\right)$, which we place to the right of all other
vertices in that cell. Similarly to the vertices of $\V$, we let
the corresponding vertex in cell $\left(1,j\right)$ of $\sigma$
be $x_{u}$. After adding these vertices, we let the resulting partially-paired
array be $\beta^{\prime\prime}$, and let the set of vertices added
to obtain $\beta^{\prime\prime}$ be $\U$. By \ref{prop:Extraction/insertion inverses},
$\beta^{\prime\prime}$ is a proper partially-paired array. Since
row $n$ of $\beta^{\prime\prime}$ has $P$ marked cells, while row
1 of $\sigma$ has $R_{n}$ marked cells, we have $\left|\U\right|=P-R_{n}$.
Also, since $\W$ is a $\left(s_{n+1}-P+R_{n}\right)$-subset, we
have $\left|\U\right|+\left|\V\right|=s_{n+1}$ as desired.

Next, we extend $\beta^{\prime\prime}$ by adding row $n+1$. For
each cell $\left(2,j\right)$ of $\sigma$ that is marked, we mark
cell $\left(n+1,j\right)$ of $\beta^{\prime\prime}$. Similarly,
for each vertex $x_{v}$ in cell $\left(2,j\right)$ of $\sigma$,
we add a corresponding vertex $v$ in row $\left(n+1,j\right)$ of
$\beta^{\prime\prime}$. Then, for each pair $\left\{ x_{u},x_{v}\right\} $
in $\sigma$, we pair their corresponding vertices $u\in\U\cup\V$
and $v$ in row $n+1$. Finally, we unmark the cells containing the
vertices of $\U$ to recover $\alpha$. By construction, $\alpha$
satisfies the parameters $\mathbf{R}$ and $\mathbf{s}$, and contains
no non-mixed pairs, so $\alpha\in\PVA{n}{K}{\mathbf{R}}{\mathbf{s}}$
as desired.

As with the other direction, we copied all marked cells and vertex
pairs of $\left(\sigma,\phi\right)$ into rows $n$ and $n+1$ of
$\alpha$. As $\sigma$ satisfies the balance condition, we have $s_{n,n+1,j}=s_{n+1,n,j}$
and $s_{n+1,k,j}=s_{k,n+1,j}=0$ for $k<n$. By \ref{lem:tree shape balance},
the fact that $\beta$ satisfies the balance condition means that
$s_{i,k,j}^{\prime}=s_{k,i,j}^{\prime}$ for all $1\le i<k\le n$
and $1\le j\le K$. Hence, $\alpha$ satisfies the balance condition.
By the compatibility condition, $\phi$ is the same as $\theta_{n}$,
the forest condition function for row $n$ of $\beta$. Furthermore,
by replacing marked cells with critical pairs of $\left(\sigma,\phi\right)$,
we have ensured that rows $n$ and $n+1$ of $\alpha$ have the same
set of marked cells and forest condition functions as rows 1 and 2
of $\left(\sigma,\phi\right)$. As the forest condition of the other
rows are unchanged, $\alpha$ is a proper vertical array.

Finally, we have to show that the two operations presented are inverses
of each other. By \ref{prop:Extraction/insertion inverses}, the extraction
and insertion procedures are inverses. Furthermore, if we extract
$\V$ and reinsert it, the vertices inserted acquire the same labels
as before the extraction. Therefore, we can correspond the non-critical
vertices in row 1 of $\left(\sigma,\phi\right)$ with the vertices
of $\V$. Then, the columns which contain the critical vertices $\U$
are exactly the columns of $\left(\sigma,\phi\right)$ that contain
critical vertices in row 1. This allows us to recover the columns
of $\U$, so that we can add critical vertices and unmark cells. Similarly,
the vertices in row 2 of $\left(\sigma,\phi\right)$ correspond to
the vertices of row $n+1$ of $\alpha$. As we have a correspondence
between the vertices of $\U\cup\V$ and vertices of row $n+1$ with
the vertices in row 1 and 2 of $\left(\sigma,\phi\right)$, respectively,
we can recover the pairing of the removed vertices via the pairing
of vertices in $\left(\sigma,\phi\right)$. Therefore, $\zeta$ as
described, is a bijection.
\end{proof}
Note that in the proof of \ref{thm:Vertical Array Decomposition},
$\alpha^{\prime}$ and $\beta^{\prime\prime}$ correspond to each
other, so does $\alpha^{\prime\prime}$ and $\beta^{\prime}$. Also,
the decomposition works with any row that is a leaf vertex of the
support graph. With this decomposition, we can iteratively pick a
row where the support graph of $\mathbf{s}$ is a leaf, and remove
that row. This leaves arrowed arrays with support graph $\mathbf{s}^{\prime}$,
which is a tree with $n$ rows, so we can repeat the process.

As an example, we will decompose the tree-shaped vertical array in
\ref{fig:3 row Vertical Array}. By following the decomposition described
in \ref{thm:Vertical Array Decomposition}, we can decompose row 3
and arrive at the partially-paired array $\alpha^{\prime}$ and $\alpha^{\prime\prime}$,
as depicted in \ref{fig:Vertical Array Labelled}. For clarity, we
have marked the vertices of $\U$ and $\V$ in $\alpha^{\prime}$,
and labelled the objects in row 2 of $\alpha^{\prime\prime}$. After
the decomposition, we obtain the minimal array $\beta$ and the arrowed
array $\left(\sigma,\phi\right)$, depicted in \ref{fig:Vertical Array Decomposition},
as well as the subset $\W=\left\{ 1,2,5\right\} \in\ksub{8}{3}$ and
the value $P=3$.

\begin{figure}
\begin{centering}
\includegraphics{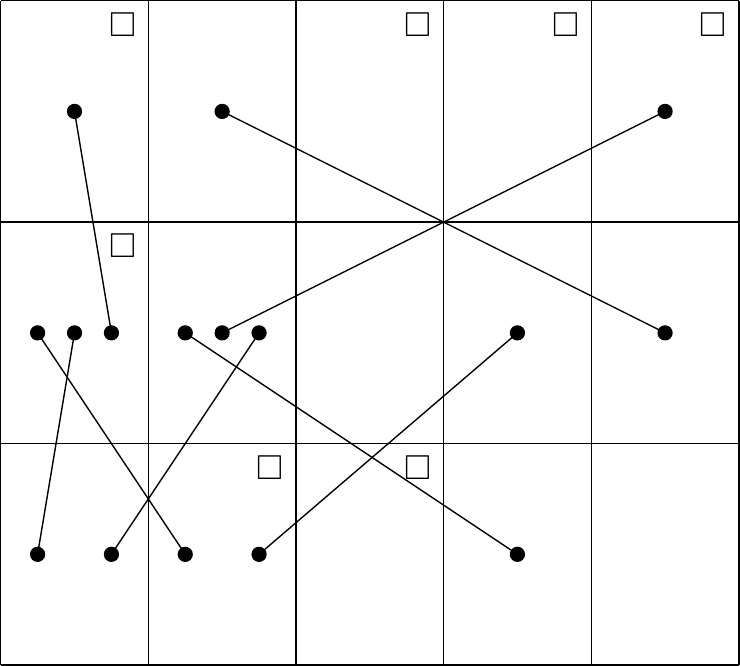}
\par\end{centering}
\caption{\label{fig:3 row Vertical Array}A tree-shaped, 3-row vertical array}
\end{figure}

\begin{figure}
\begin{centering}
\includegraphics{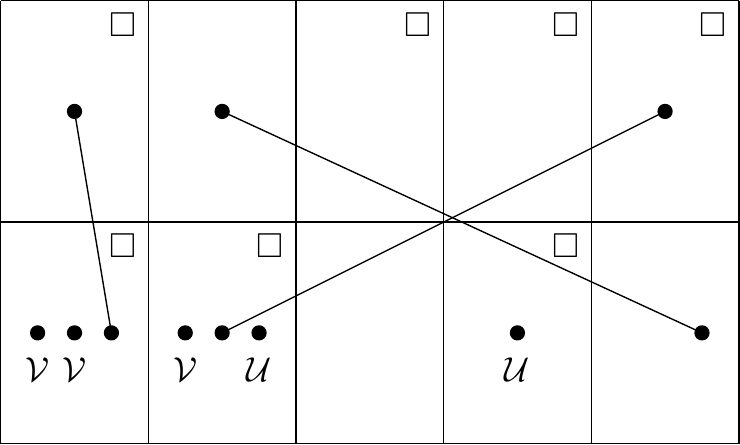}$\qquad$\includegraphics{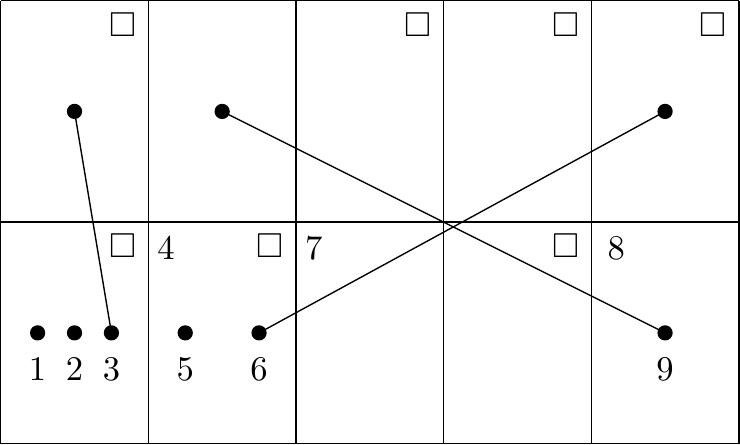}
\par\end{centering}
\caption{\label{fig:Vertical Array Labelled}Partially-paired array $\alpha^{\prime}$
and $\alpha^{\prime\prime}$ corresponding to the decomposition of
row 3 of \ref{fig:3 row Vertical Array}}
\end{figure}

\begin{figure}
\begin{centering}
\includegraphics{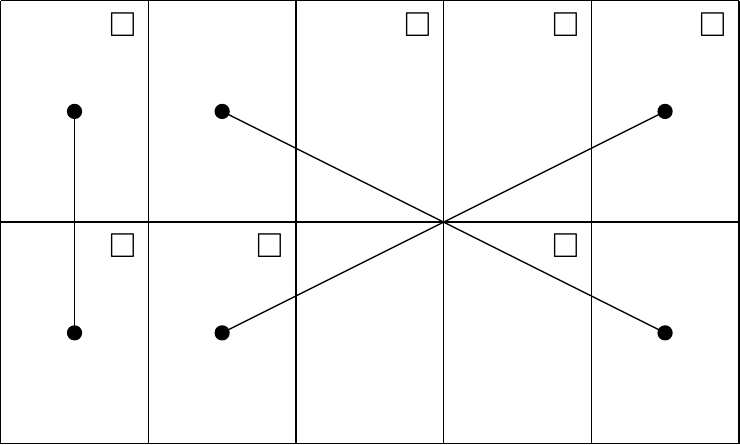}$\qquad$\includegraphics{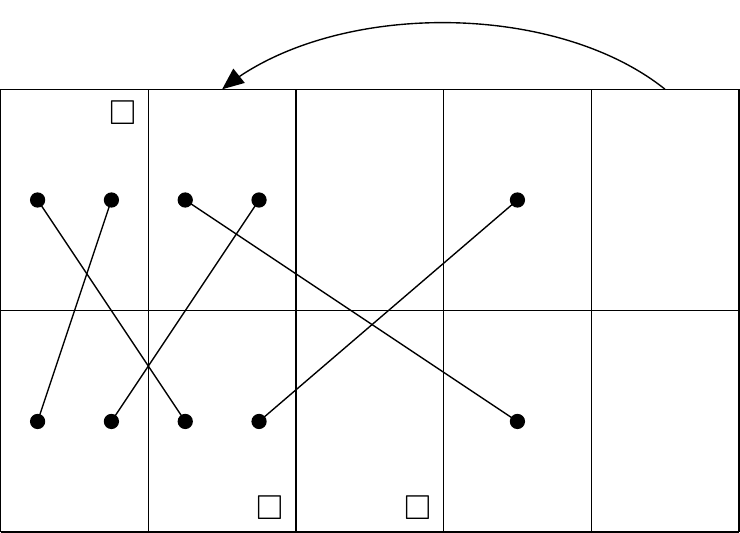}
\par\end{centering}
\caption{\label{fig:Vertical Array Decomposition}Vertical array $\beta$ and
arrowed array $\left(\sigma,\phi\right)$ corresponding to the decomposition
of row 3 of \ref{fig:3 row Vertical Array}}
\end{figure}

Now that we have a decomposition of tree-shaped vertical arrays, we
can provide an explicit formula for $v_{n,K;\mathbf{R}}^{\left(\mathbf{s}\right)}$
via induction. We start with the following corollary.
\begin{cor}
\label{cor:Vertical Array Decomposition}Let $n,K\ge1$, $\mathbf{s}=\left(s_{1,2},s_{1,3},\dots,s_{n,n+1}\right)\ge\mathbf{0}$,
and $\mathbf{R}=\left(R_{1},\dots,R_{n+1}\right)\in\left[K\right]^{n+1}$.
Suppose the support graph of $\mathbf{s}$ is a tree with the vertex
$n+1$ as a leaf adjacent to the vertex $n$. Then, 
\begin{eqnarray*}
v_{n+1,K;\mathbf{R}}^{\left(\mathbf{s}\right)} & = & \sum_{P=R_{n}}^{\min\left(s_{n+1}+R_{n},K\right)}\sum_{A_{n+1}=0}^{\min\left(s_{n+1},K\right)-1}\binom{s_{n}+R_{n}-1}{s_{n+1}-P+R_{n}}v_{n,K;\mathbf{R}^{\prime}}^{\left(\mathbf{s}^{\prime}\right)}\times\\
 &  & \frac{\left(s_{n+1}-P+R_{n}\right)\left(K-A_{n+1}-1\right)!\left(s_{n+1}-A_{n+1}-1\right)!\left(P-1\right)!}{\left(P-R_{n}-A_{n+1}\right)!\left(K-R_{n+1}-A_{n+1}\right)!\left(R_{n}-1\right)!\left(R_{n+1}-1\right)!}
\end{eqnarray*}
where $\mathbf{s}^{\prime}$ is $\mathbf{s}$ restricted to an $n\times n$
matrix by removing the last row and column, $s_{i}=\sum_{k\neq i}s_{i,k}$
for $1\le i\le n+1$, and $\mathbf{R}^{\prime}$ is a vector of length
$n$ given by 
\begin{eqnarray*}
R_{k}^{\prime} & = & \begin{cases}
R_{k} & k<n\\
P & k=n
\end{cases}
\end{eqnarray*}
\end{cor}
\begin{proof}
Let $P$ be such that $R_{n}\le P\le\min\left(s_{n+1}+R_{n},K\right)$,
$\beta\in\PVA{n}{K}{\mathbf{R}^{\prime}}{\mathbf{s}^{\prime}}$ be
an $n$-row vertical array with parameters as defined in \ref{thm:Vertical Array Decomposition},
$\W$ be a $\left(s_{n+1}-P+R_{n}\right)$-subset of $\left[s_{n}+R_{n}-1\right]$,
and $\Lambda_{\beta,n,\W}$ be a substructure of $\AR{K}{R_{n}}{R_{n+1}}{s_{n+1}}$.
As $\beta$ is a proper vertical array, the forest condition function
$\theta_{n}$ for row $n$ is a forest with root vertices $\P=\R_{n}^{\prime}$.
Therefore, by applying \ref{thm:Substructure Lambda Formula}, we
have
\begin{eqnarray*}
T\left(\Lambda_{\beta,n,\W}\right) & = & \sum_{A_{n+1}=0}^{\min\left(s_{n+1},K\right)-1}\frac{\left(s_{n+1}-P+R_{n}\right)\left(K-A_{n+1}-1\right)!\left(s_{n+1}-A_{n+1}-1\right)!\left(P-1\right)!}{\left(P-R_{n}-A_{n+1}\right)!\left(K-R_{n+1}-A_{n+1}\right)!\left(R_{n}-1\right)!\left(R_{n+1}-1\right)!}
\end{eqnarray*}
Note that this formula is independent of $\beta$ and $\W$, and only
depends on $P$. Furthermore, the constraint $R_{n}\le P\le\min\left(s_{n+1}+R_{n},K\right)$
matches with the definition of substructure $\Lambda$. Then, for
a given $P$, there are $\binom{s_{n}+R_{n}-1}{s_{n+1}-P+R_{n}}$
distinct $\left(s_{n+1}-P+R_{n}\right)$-subsets of $\left[s_{n}+R_{n}-1\right]$.
Finally, for a given $R_{n}^{\prime}=P$, there are $v_{n,K;\mathbf{R}^{\prime}}^{\left(\mathbf{s}^{\prime}\right)}$
proper vertical arrays. Combining these gives the formula of our corollary
as desired.
\end{proof}
As we have assumed that the support graph $G$ of $\mathbf{s}$ is
a tree, we can repeatedly select a row that corresponds to a leaf
vertex in $G$, and iterate the decomposition in \ref{thm:Vertical Array Decomposition}.
Then, by taking the cardinality of both sides, we obtain the following
theorem.
\begin{thm}
\label{thm:Vertical Array Formula}Let $n,K\ge1$, $\mathbf{s}\ge\mathbf{0}$,
and $\mathbf{R}\ge\mathbf{1}$. Suppose the support graph $G$ of
$\mathbf{s}$ is a tree. Then,
\begin{eqnarray*}
v_{n,K;\mathbf{R}}^{\left(\mathbf{s}\right)} & = & \sum_{A_{e_{1}}=0}^{\min\left(s_{e_{1}},K\right)-1}\cdots\sum_{A_{e_{n-1}}=0}^{\min\left(s_{e_{n-1}},K\right)-1}\left[\prod_{j=1}^{n-1}\frac{\left(K-A_{e_{j}}-1\right)!}{\left(K+s_{e_{j}}-A_{e_{j}}-1\right)!}\times\right.\\
 &  & \left.\prod_{i=1}^{n}\frac{\left(K+\sum_{k\sim i}\left(s_{i,k}-A_{i,k}-1\right)\right)!\left(R_{i}-1+\sum_{k\sim i}s_{i,k}\right)!}{\left(R_{i}-1\right)!\left(K-R_{i}-\sum_{k\sim i}A_{i,k}\right)!\left(R_{i}+\sum_{k\sim i}\left(s_{i,k}-1\right)\right)!}\right]
\end{eqnarray*}
where $e_{1},\dots,e_{n-1}$ are the edges of $G$. Furthermore, for
each edge $e_{j}=\left\{ i,k\right\} $ in $G$, the summation variable
$A_{e_{j}}$ can be equivalently written as $A_{i,k}$ and $A_{k,i}$.
Finally, the sum $\sum_{k\sim i}$ is over all indices $k$ that are
adjacent to $i$ in the support graph of $\mathbf{s}$.
\end{thm}
For example, if $n=3$ and $s_{1,2}=0$, the formula reduces to
\begin{eqnarray}
v_{n,K;\mathbf{R}}^{\left(\mathbf{s}\right)} & = & \sum_{A_{1,3}=0}^{\min\left(s_{1,3},K\right)-1}\sum_{A_{2,3}=0}^{\min\left(s_{2,3},K\right)-1}\left[\frac{\left(K-A_{1,3}-1\right)!\left(K-A_{2,3}-1\right)!}{\left(R_{1}-1\right)!\left(R_{2}-1\right)!\left(R_{3}-1\right)!\left(K-R_{1}-A_{1,3}\right)!}\times\right.\nonumber \\
 &  & \left.\frac{\left(K+s_{1,3}+s_{2,3}-A_{1,3}-A_{2,3}-2\right)!\left(R_{3}+s_{1,3}+s_{2,3}-1\right)!}{\left(K-R_{2}-A_{2,3}\right)!\left(K-R_{3}-A_{1,3}-A_{2,3}\right)!\left(R_{3}+s_{1,3}+s_{2,3}-2\right)!}\right]\label{eq:v3 example}
\end{eqnarray}
As we shall later see, we can remove the upper bounds of $K-1$, but
upper bounds of $s_{e_{j}}-1$ are necessary and cannot be removed.
\begin{proof}
Note that if $R_{i}>K$ for some $i$, the term $\left(K-R_{i}-\sum_{k\sim i}A_{i,k}\right)!$
in the denominator causes the entire sum to be zero. Otherwise, we
prove this theorem via induction on the number of rows.

\textbf{Base case:}

Suppose $n=1$, then we have no vertices, so there are $\binom{K}{R_{1}}$
vertical arrays in $\PVA{1}{K}{R_{1}}{\mathbf{s}}$. This matches
with our formula for $v_{1,K;R_{1}}^{\left(\mathbf{s}\right)}$, as
the variables $s_{e_{j}}$ and summations $A_{e_{j}}$ do not appear.

\textbf{Inductive step:}

Let $\mathbf{s}=\left(s_{1,2},s_{1,3},\dots,s_{n,n+1}\right)\ge\mathbf{0}$,
$\mathbf{R}=\left(R_{1},\dots,R_{n+1}\right)\ge\mathbf{1}$, and assume
that the functional digraph of $\mathbf{s}$ has vertex $n+1$ as
a leaf, and is adjacent to vertex $n$. Then, for convenience of notation,
let $s_{i}=\sum_{k\sim i}s_{i,k}$, $A_{i}=\sum_{k\sim i}A_{i,k}$,
and $\delta_{i}=\sum_{k\sim i}1$ for $1\le i\le n+1$. Furthermore,
let the $e_{1},\dots,e_{n}$ be the edges of the support graph of
$\mathbf{s}$, with $e_{n}$ being the edge between vertex $n$ and
$n+1$. This means that $s_{n+1}=s_{n,n+1}$, $A_{n+1}=A_{n,n+1}$,
and for $1\le i\le n-1$, $A_{i}$ does not contain the variable $A_{n,n+1}$.

By applying $s_{n}^{\prime}=s_{n}-s_{n+1}$ to our inductive hypothesis,
we have

\begin{eqnarray*}
v_{n,K;\mathbf{R}^{\prime}}^{\left(\mathbf{s}^{\prime}\right)} & = & \sum_{A_{e_{1}}=0}^{\min\left(s_{e_{1}},K\right)-1}\cdots\sum_{A_{e_{n-1}}=0}^{\min\left(s_{e_{n-1}},K\right)-1}\left[\prod_{j=1}^{n-1}\frac{\left(K-A_{e_{j}}-1\right)!}{\left(K+s_{e_{j}}-A_{e_{j}}-1\right)!}\times\right.\\
 &  & \prod_{i=1}^{n-1}\frac{\left(K+s_{i}-A_{i}-\delta_{i}\right)!\left(R_{i}+s_{i}-1\right)!}{\left(R_{i}-1\right)!\left(K-R_{i}-A_{i}\right)!\left(R_{i}+s_{i}-\delta_{i}\right)!}\times\\
 &  & \left.\frac{\left(K+s_{n}-s_{n+1}-A_{n}+A_{n+1}-\delta_{n}+1\right)!\left(P+s_{n}-s_{n+1}-1\right)!}{\left(P-1\right)!\left(K-P-A_{n}+A_{n+1}\right)!\left(P+s_{n}-s_{n+1}-\delta_{n}+1\right)!}\right]
\end{eqnarray*}
Note that $A_{n}$ and $\delta_{n}$ are substituted with $A_{n}-A_{n+1}$
and $\delta_{n}-1$, respectively, as the support graph of $\mathbf{s}^{\prime}$
does not contain the edge $e_{n}=\left\{ n,n+1\right\} $. To simplify
the expression for further manipulation, we let $C\left(A_{e_{1}},\dots,A_{e_{n-1}}\right)$
to be the first two products inside the sum. That is, we rewrite the
above expression as 
\begin{eqnarray*}
v_{n,K;\mathbf{R}^{\prime}}^{\left(\mathbf{s}^{\prime}\right)} & = & \sum_{A_{e_{1}},\dots,A_{e_{n-1}}}C\left(A_{e_{1}},\dots,A_{e_{n-1}}\right)\times\\
 &  & \frac{\left(K+s_{n}-s_{n+1}-A_{n}+A_{n+1}-\delta_{n}+1\right)!\left(P+s_{n}-s_{n+1}-1\right)!}{\left(P-1\right)!\left(K-P-A_{n}+A_{n+1}\right)!\left(P+s_{n}-s_{n+1}-\delta_{n}+1\right)!}
\end{eqnarray*}
Then, we can substitute this expression into \ref{cor:Vertical Array Decomposition},
which gives
\begin{eqnarray*}
v_{n+1,K;\mathbf{R}}^{\left(\mathbf{s}\right)} & = & \sum_{A_{e_{1}},\dots,A_{e_{n-1}}}\sum_{A_{n+1}=0}^{\min\left(s_{n+1},K\right)-1}\sum_{P=0}^{\min\left(s_{n+1},K-R_{n}\right)}C\left(A_{e_{1}},\dots,A_{e_{n-1}}\right)\times\\
 &  & \frac{\left(K+s_{n}-s_{n+1}-A_{n}+A_{n+1}-\delta_{n}+1\right)!\left(s_{n}+R_{n}-1\right)!}{\left(K-P-R_{n}-A_{n}+A_{n+1}\right)!\left(P+R_{n}+s_{n}-s_{n+1}-\delta_{n}+1\right)!}\times\\
 &  & \frac{\left(K-A_{n+1}-1\right)!\left(s_{n+1}-A_{n+1}-1\right)!}{\left(s_{n+1}-P-1\right)!\left(P-A_{n+1}\right)!\left(K-R_{n+1}-A_{n+1}\right)!\left(R_{n}-1\right)!\left(R_{n+1}-1\right)!}
\end{eqnarray*}
after shifting the summation index $P$ down by $R_{n}$. As the terms
$\left(s_{n+1}-P-1\right)!$ and $\left(K-P-R_{n}-A_{n}+A_{n+1}\right)!$
are in the denominator, the summation term is zero if $P>\min\left(s_{n+1},K-R_{n}\right)$,
noting that $A_{n}\ge A_{n+1}$. Hence, we can safely increase the
upper bound of the $P$ summation to infinity. Furthermore, if $P<A_{n+1}$,
then the summation term is also zero, as we have $\left(P-A_{n+1}\right)!$
in the denominator. This allows us to substitute $P=Q+A_{n+1}$ and
sum over $Q\ge0$ instead. By doing these substitutions, we obtain
\begin{eqnarray*}
v_{n+1,K;\mathbf{R}}^{\left(\mathbf{s}\right)} & = & \sum_{A_{e_{1}},\dots,A_{e_{n-1}}}\sum_{A_{n+1}=0}^{\min\left(s_{n+1},K\right)-1}\sum_{Q\ge0}C\left(A_{e_{1}},\dots,A_{e_{n-1}}\right)\times\\
 &  & \frac{\left(K+s_{n}-s_{n+1}-A_{n}+A_{n+1}-\delta_{n}+1\right)!\left(s_{n}+R_{n}-1\right)!}{\left(K-Q-R_{n}-A_{n}\right)!\left(Q+R_{n}+A_{n+1}+s_{n}-s_{n+1}-\delta_{n}+1\right)!}\times\\
 &  & \frac{\left(K-A_{n+1}-1\right)!\left(s_{n+1}-A_{n+1}-1\right)!}{\left(s_{n+1}-A_{n+1}-Q-1\right)!Q!\left(K-R_{n+1}-A_{n+1}\right)!\left(R_{n}-1\right)!\left(R_{n+1}-1\right)!}\\
 & = & \sum_{A_{e_{1}},\dots,A_{e_{n-1}}}\sum_{A_{n+1}=0}^{\min\left(s_{n+1},K\right)-1}C\left(A_{e_{1}},\dots,A_{e_{n-1}}\right)\times\\
 &  & \frac{\left(K+s_{n}-A_{n}-\delta_{n}\right)!\left(R_{n}+s_{n}-1\right)!\left(K-A_{n+1}-1\right)!}{\left(R_{n}-1\right)!\left(K-R_{n}-A_{n}\right)!\left(R_{n}+s_{n}-\delta_{n}\right)!\left(K-R_{n+1}-A_{n+1}\right)!\left(R_{n+1}-1\right)!}
\end{eqnarray*}
using the Chu-Vandermonde identity (pg. 67 of \cite{Andrews-Askey-Roy:1999}).
By noting that $\delta_{n+1}=1$ and simplifying the formula we obtained
for $v_{n+1,K;\mathbf{R}}^{\left(\mathbf{s}\right)}$, we obtain
\begin{eqnarray*}
v_{n+1,K;\mathbf{R}}^{\left(\mathbf{s}\right)} & = & \sum_{A_{e_{1}}=0}^{\min\left(s_{e_{1}},K\right)-1}\cdots\sum_{A_{e_{n}}=0}^{\min\left(s_{e_{n}},K\right)-1}\left[\prod_{j=1}^{n}\frac{\left(K-A_{e_{j}}-1\right)!}{\left(K+s_{e_{j}}-A_{e_{j}}-1\right)!}\times\right.\\
 &  & \left.\prod_{i=1}^{n+1}\frac{\left(K+s_{i}-A_{i}-\delta_{i}\right)!\left(R_{i}+s_{i}-1\right)!}{\left(R_{i}-1\right)!\left(K-R_{i}-A_{i}\right)!\left(R_{i}+s_{i}-\delta_{i}\right)!}\right]
\end{eqnarray*}
which proves our induction as desired.
\end{proof}
To remove the upper bounds of $K-1$ in \ref{thm:Vertical Array Formula},
we will for each edge $e$ of the support graph of $\mathbf{s}$,
assign a vertex $v$ that is incident to $e$. This will allow us
to regroup the factorial terms in $v_{n,K;\mathbf{R}}^{\left(\mathbf{s}\right)}$,
which will allow us to rewrite the expression with rising factorials.
\begin{cor}
\label{cor:Vertical Array Formula}Let $n,K\ge1$, $\mathbf{s}\ge\mathbf{0}$,
and $\mathbf{R}\ge\mathbf{1}$. Suppose that the support graph $G$
of $\mathbf{s}$ is a tree with edges $e_{1},\dots,e_{n-1}$, such
that $e_{j}$ is incident with vertex $j$ in $G$ for $1\le j\le n-1$.
Then,

\begin{eqnarray}
v_{n,K;\mathbf{R}}^{\left(\mathbf{s}\right)} & = & \prod_{i=1}^{n}\frac{\left(R_{i}-1+\sum_{k\sim i}s_{i,k}\right)!}{\left(R_{i}-1\right)!\left(R_{i}+\sum_{k\sim i}\left(s_{i,k}-1\right)\right)!}\times\nonumber \\
 &  & \sum_{A_{e_{1}}=0}^{s_{e_{1}}-1}\cdots\sum_{A_{e_{n-1}}=0}^{s_{e_{n-1}}-1}\left(K-R_{n}-\sum_{k\sim n}A_{n,k}+1\right)^{\left(\sum_{k\sim n}\left(s_{n,k}-1\right)+R_{n}\right)}\times\nonumber \\
 &  & \prod_{j=1}^{n-1}\left[\left(K-R_{j}-\sum_{k\sim j}A_{j,k}+1\right)^{\left(R_{j}+\sum_{k\sim j}A_{j,k}-A_{e_{j}}-1\right)}\times\right.\nonumber \\
 &  & \left.\left(K+s_{e_{j}}-A_{e_{j}}\right)^{\left(\sum_{k\sim j}\left(s_{j,k}-A_{j,k}-1\right)-s_{e_{j}}+A_{e_{j}}+1\right)}\right]\label{eq:Corollary equation}
\end{eqnarray}
where for each edge $e_{j}=\left\{ j,\ell\right\} $ in $G$, the
summation variable $A_{e_{j}}$ can be equivalently written as $A_{j,\ell}$
and $A_{\ell,j}$. As in \ref{thm:Vertical Array Formula}, the sum
$\sum_{k\sim j}$ is over all indices $k$ that are adjacent to $j$
in the support graph of $\mathbf{s}$. Furthermore, $v_{n,K;\mathbf{R}}^{\left(\mathbf{s}\right)}$
as expressed in this corollary is a polynomial in $K$.
\end{cor}
Note that given an arbitrary tree with vertices $1,\dots,n$, we can
achieve the incidence condition in \ref{cor:Vertical Array Formula}
by repeatively taking a leaf vertex $k\neq n$, label the edge incident
to vertex $k$ as $e_{k}$, then delete both vertex $k$ and edge
$e_{k}$.
\begin{proof}
First, we rearrange the expression for $v_{n,K;\mathbf{R}}^{\left(\mathbf{s}\right)}$
in \ref{thm:Vertical Array Formula} to obtain
\begin{eqnarray}
v_{n,K;\mathbf{R}}^{\left(\mathbf{s}\right)} & = & \prod_{i=1}^{n}\frac{\left(R_{i}-1+\sum_{k\sim i}s_{i,k}\right)!}{\left(R_{i}-1\right)!\left(R_{i}+\sum_{k\sim i}\left(s_{i,k}-1\right)\right)!}\times\nonumber \\
 &  & \sum_{A_{e_{1}}=0}^{\min\left(s_{e_{1}},K\right)-1}\cdots\sum_{A_{e_{n-1}}=0}^{\min\left(s_{e_{n-1}},K\right)-1}\frac{\left(K+\sum_{k\sim n}\left(s_{n,k}-A_{n,k}-1\right)\right)!}{\left(K-R_{n}-\sum_{k\sim n}A_{n,k}\right)!}\times\nonumber \\
 &  & \prod_{j=1}^{n-1}\left[\frac{\left(K-A_{e_{j}}-1\right)!}{\left(K-R_{j}-\sum_{k\sim j}A_{j,k}\right)!}\cdot\frac{\left(K+\sum_{k\sim j}\left(s_{j,k}-A_{j,k}-1\right)\right)!}{\left(K+s_{e_{j}}-A_{e_{j}}-1\right)!}\right]\label{eq:Proof equation}
\end{eqnarray}
As the sums $\sum_{k\sim j}$ are over the support graph of $\mathbf{s}$,
and $e_{j}$ is incident to vertex $j$, we have $\sum_{k\sim n}\left(s_{n,k}-1\right)\ge0$,
$A_{e_{j}}\le\sum_{k\sim j}A_{j,k}$, and $s_{e_{j}}-A_{e_{j}}-1\le\sum_{k\sim j}\left(s_{j,k}-A_{j,k}-1\right)$.
Hence, we can rewrite the ratios of factorials in rows 2 and 3 into
rising factorials. This gives us the expression in \ref{eq:Corollary equation},
but still retaining the upper bounds of $K-1$. It remains to show
that the summation term is equal to zero if $A_{e_{j}}\ge K$ for
some edge $e_{j}$.

Now, suppose $0\le A_{e_{\ell}}\le s_{e_{\ell}}-1$ holds for all
edges $e_{\ell}$, but there exists some edge $e_{j}$ such that $A_{e_{j}}\ge K$.
Let $G^{\prime}=\left(V,E^{\prime}\right)$ be the subgraph of $G$
such that $\left\{ i,k\right\} \in E^{\prime}$ if and only if $A_{i,k}\ge K$.
As $G^{\prime}$ is a forest, there must a vertex $\ell$ that is
incident to some edge $\left\{ j,\ell\right\} \in G^{\prime}$, but
$e_{\ell}\notin G^{\prime}$. If $\ell=n$, then
\begin{eqnarray*}
\left(K-R_{n}-\sum_{k\sim n}A_{n,k}+1\right)^{\left(\sum_{k\sim n}\left(s_{n,k}-1\right)+R_{n}\right)} & = & \frac{\left(K+\sum_{k\sim n}\left(s_{n,k}-A_{n,k}-1\right)\right)!}{\left(K-R_{n}-\sum_{k\sim n}A_{n,k}\right)!}\\
 & = & 0
\end{eqnarray*}
as $A_{n,k}\le s_{n,k}-1$ and $K-\sum_{k\sim n}A_{n,k}\le K-A_{n,\ell}\le0$.
Otherwise, we have $\ell\neq n$, which yields 
\begin{eqnarray*}
\left(K-R_{\ell}-\sum_{k\sim\ell}A_{\ell,k}+1\right)^{\left(R_{\ell}+\sum_{k\sim\ell}A_{\ell,k}-A_{e_{\ell}}-1\right)} & = & \frac{\left(K-A_{e_{\ell}}-1\right)!}{\left(K-R_{\ell}-\sum_{k\sim\ell}A_{\ell,k}\right)!}\\
 & = & 0
\end{eqnarray*}
as $e_{\ell}\notin G^{\prime}$ implies that $A_{e_{\ell}}\le K-1$,
and $\left\{ j,\ell\right\} \in G^{\prime}$ implies that $K-\sum_{k\sim\ell}A_{\ell,k}\le K-A_{\ell,j}\le0$.
In both cases, at least one of the rising factorials is zero within
the summation term, so the entire term is zero if $A_{e_{j}}\ge K$.

Finally, since the number of terms in the rising factorials is independent
of $K$, and the number of summation terms is determined by the $s_{e_{j}}$'s,
the expression for $v_{n,K;\mathbf{R}}^{\left(\mathbf{s}\right)}$
as written in this corollary is a polynomial in $K$, as desired.
\end{proof}
For example, suppose $n=3$ and $s_{1,2}=0$. Then, our formula for
$v_{n,K;\mathbf{R}}^{\left(\mathbf{s}\right)}$ in \ref{eq:v3 example}
can be written as
\begin{eqnarray*}
v_{n,K;\mathbf{R}}^{\left(\mathbf{s}\right)} & = & \frac{\left(R_{1}+s_{1,3}+s_{2,3}-1\right)!}{\left(R_{1}-1\right)!\left(R_{2}-1\right)!\left(R_{3}-1\right)!\left(R_{1}+s_{1,3}+s_{2,3}-2\right)!}\times\\
 &  & \sum_{A_{1,3}=0}^{s_{1,3}-1}\sum_{A_{2,3}=0}^{s_{2,3}-1}\left(K-R_{3}-A_{1,3}-A_{2,3}+1\right)^{\left(R_{3}+s_{1,3}+s_{2,3}-2\right)}\times\\
 &  & \left[\left(K-R_{1}-A_{1,3}+1\right)^{\left(R_{1}-1\right)}\left(K-R_{2}-A_{2,3}+1\right)^{\left(R_{2}-1\right)}\right]
\end{eqnarray*}
which is a polynomial in $K$.

With this corollary, we have obtained an expression for $v_{n,K;\mathbf{R}}^{\left(\mathbf{s}\right)}$
that is a polynomial in $K$ for all $\mathbf{R\ge1}$, if the support
graph of $\mathbf{s}$ is a tree. We can substitute this into \ref{thm:Canonical Vertical Formula}
to obtain a polynomial expression for $f_{n,K}^{\left(\mathbf{q};\mathbf{s}\right)}$
by \ref{thm:Label to Canonical Array}. Then, using \ref{eq:Genfunc Label Relation},
we can substitute $K=x$ into the expression for $f_{n,K}^{\left(\mathbf{q};\mathbf{s}\right)}$
to obtain $A_{n}^{\left(\mathbf{q};\mathbf{s}\right)}\left(x\right)$,
proving \ref{thm:main formula} as desired.

\section{Acknowledgements}

Many thanks for the help of I.P. Goulden for supporting me in my doctoral
studies, during which this research is conducted, as well as the editing
and verifying of the results in this paper.

\bibliographystyle{plain}
\bibliography{uw-ethesis}

\end{document}